\documentclass[12pt]{amsart}
\usepackage{amssymb}
\usepackage{amsmath}
\usepackage{verbatim}
\usepackage[usenames]{color}
\usepackage{hyperref}
\usepackage{url}
\usepackage{multicol}
\usepackage{tikz}
\usepackage{tipa}
\usepackage{nicefrac}
\usepackage{enumerate}

\newtheorem{thm}{Theorem}
 
\newtheorem{la}[thm]{Lemma}
\newtheorem{mainlem}[thm]{Main Lemma}
\newtheorem{cor}[thm]{Corollary}
\newtheorem{sclaim}{Claim}[thm]

\theoremstyle{definition}
\newtheorem{df}[thm]{Definition}

\newtheorem{que}[thm]{Question}

\theoremstyle{remark}

\newenvironment{ls}{\begin{itemize}}{\end{itemize}}

\newenvironment{pf}{\begin{proof}}{\end{proof}}
\newcommand{\ger}[1]{\ensuremath{\mathfrak {#1}}}
\newcommand{\scr}[1]{\ensuremath{\mathcal {#1}}}
\newcommand{\bld}[1]{\ensuremath{\mathbf {#1}}}
\newcommand{\bbb}[1]{\ensuremath{\mathbb {#1}}}

\renewcommand{\phi}{\varphi}

\newcommand{\set}[2]{\ensuremath{\{#1:#2\}}}

\newcommand{\notarrow}{\kern .42em\not\kern -.42em\longrightarrow}
\renewcommand{\th}{\ensuremath{{}^{\text{th}}}}

\newcommand{\rr}{\ensuremath{\ger{rr}}}
\newcommand{\rrf}{\ensuremath{\ger{rr}_f}}
\newcommand{\rri}{\ensuremath{\ger{rr}_i}}
\newcommand{\rro}{\ensuremath{\ger{rr}_o}}
\newcommand{\ssi}{\ensuremath{\ss_i}}
\newcommand{\sso}{\ensuremath{\ss_{o}}}
\newcommand{\salmost}{\ensuremath{\ger s_{\mathrm{almost}}}}
\newcommand{\stotal}{\ensuremath{\ger s_{\mathrm{total}}}}
\newcommand{\sr}{\ensuremath{\ger{sr}}}
\newcommand{\sri}{\ensuremath{\sr_i}}
\newcommand{\sro}{\ensuremath{\sr_{o}}}
\newcommand{\cov}[1]{\ensuremath{\bld{cov}(\scr{#1})}}
\newcommand{\non}[1]{\ensuremath{\bld{non}(\scr{#1})}}

\newcommand{\N}{\mathbb N}

\newcommand{\LL}{\mathbb L}
\newcommand{\stem}{{\mathrm{stem}}}
\newcommand{\sub}{\subseteq}
\newcommand{\ha}{\,{}\hat{}\,}
\newcommand{\sem}{\setminus}
\newcommand{\omoms}{[\omega]^\omega}
\newcommand{\forces}{\Vdash}
\newcommand{\omlom}{\omega^{<\omega}}
\newcommand{\rk}{{\mathrm{rk}}}
\newcommand{\re}{{\upharpoonright}}
\renewcommand{\P}{{\mathcal P}}
\newcommand{\PP}{{\mathbb P}}
\newcommand{\QQ}{{\mathbb{Q}}}
\newcommand{\B}{{\mathcal B}}
\newcommand{\omom}{\omega^\omega}
\renewcommand{\succ}{\mathrm{succ}}

\newcommand{\noprint}[1]{\relax}

\title{The Subseries Number}
\author[Brendle]{J\"org Brendle}
\address[J.~Brendle]{Graduate School of System Informatics, Kobe
  University, 1--1   Rokkodai, Nada-ku, 657-8501 Kobe,  Japan} 
\email{brendle@kobe-u.ac.jp}
\author[Brian]{Will Brian}
\address[W.~Brian]{Department of Mathematics and Statistics, University of North Carolina at Charlotte, 9201 University City Blvd., Charlotte, NC 28223-0001, U.S.A.}
\email{wbrian.math@gmail.com}
\urladdr{http://wrbrian.wordpress.com}
\author[Hamkins]{Joel David Hamkins}
\address[J.~D.~Hamkins]{Mathematics, The Graduate Center of the City
  Univeristy of New York, 365 Fifth Avenue, New York, NY 10016,
  U.S.A. and Mathematics, College of Staten Island of CUNY, Staten
  Island, NY 10314, U.S.A.}
\email{jhamkins@gc.cuny.edu}
\urladdr{http://jdh.hamkins.org}

\thanks{The authors wish to thank Jonathan Verner for
  several insightful discussions about the present topic, and Andreas Blass for his comments on an early version of this paper. 
}

\begin{document}

\begin{abstract}

Every conditionally convergent series of real numbers has a divergent subseries.
  How many subsets of the natural numbers are needed so that
  every conditionally convergent series diverges on the subseries corresponding to one of these sets? 
  The answer to this question is defined to be the subseries number, a new cardinal characteristic of the continuum. This cardinal is bounded below by $\aleph_1$ and above by the cardinality of the continuum, but it is not provably equal to either. We define three natural variants of the subseries number, and compare them with each other, with their corresponding rearrangement numbers, and with several well-studied cardinal characteristics of the continuum. Many consistency results are obtained from these comparisons, and we obtain another by computing the value of the subseries number in the Laver model.
\end{abstract}

\maketitle

\section{Introduction}          \label{intro}

Let $\sum_{n \in \N} a_n$ be a convergent series of real numbers. The series is conditionally convergent if, and only if, there is some $A \subseteq \N$ such that the subseries $\sum_{n \in A}a_n$ is no longer convergent. In other words, conditionally convergent series always admit divergent subseries, and this property characterizes exactly those convergent series that converge conditionally.

Given a convergent series $\sum_{n \in \N}a_n$, one may view each $A \subseteq \N$ as a test of whether the convergence of the series is conditional: if $\sum_{n \in A}a_n$ converges then the series has passed the test, and we learn nothing of whether its convergence is conditional or absolute, but if $\sum_{n \in A}a_n$ diverges then our test has revealed that the original series is only conditionally convergent. How large a battery of tests of this kind do we need so that every conditionally convergent series is revealed to be conditionally convergent by one of these tests? 

This paper explores this question, along with several related questions.  We begin, in
Section~\ref{sec:defs}, by defining the \emph{subseries number}, the
minimal cardinality of a family $\mathcal A$ of subsets of $\N$ needed to ensure that for every conditionally
convergent series $\sum_{n \in \N}a_n$, the subseries $\sum_{n \in A}a_n$ diverges for some $A \in \mathcal A$. We will define two other cardinal numbers in a similar fashion, by requiring the subseries in question to diverge in some particular way, whether by increasing or decreasing without bound, or by oscillation.

The subseries numbers are related to the rearrangement numbers, which were defined and explored in \cite{Hardy} and \cite{rearrangement}. In the present work, we will see how the subseries numbers relate to the rearrangment numbers, and we will find bounds for the subseries numbers in terms of other classical cardinal characteristics. We will use these bounds to separate, in some cases, the subseries numbers from each other, from their corresponding rearrangement numbers, from classical cardinal characteristics of the continuum, and from $\ger c$.

In the next section we will define the three subseries numbers, prove some basic facts about them, and summarize the results of this paper.

\section{Definitions, basic facts, and a summary of results} \label{sec:defs} 

We denote the \emph{subseries number} by \ss, defined as follows.

\begin{df}      \label{df-ss}
  $\ss$ is the smallest cardinality of any family $\mathcal A$ of
      subsets of \bbb N such that, for every conditionally
      convergent series $\sum_{n \in \N}a_n$ of real numbers, there is some
      $A \in \mathcal A$ such that the subseries $\sum_{n \in A}a_n$ diverges.
\end{df}

In this definition, the subseries might diverge to $+\infty$ or to
$-\infty$, or it might diverge by oscillation.  If we specify one of
these options, we get more specific subseries numbers.

\begin{df}      \label{df-rrsub}
\hskip1pt
\begin{ls}
\item $\ssi$ is defined like \ss\ except that $\sum_{n \in A}a_n$ is
    required to diverge to $\infty$ or to $-\infty$.
\item \sso\ is defined like \ss\ except that $\sum_{n \in A}a_n$ is
    required to diverge by oscillation.
\end{ls}
\end{df}


The definition of $\ss$ was first suggested by the third author on MathOverflow \cite{Hamkins}. 
For convenience, let us recall here the definitions of two of the rearrangement numbers as well:

\begin{df}      \label{df-rr}
$\ $
\begin{ls}
  \item   \rr\ is the smallest cardinality of any family $\mathcal C$ of
      permutations of \bbb N such that, for every conditionally
      convergent series $\sum_na_n$ of real numbers, there is some
      permutation $p\in \mathcal C$ for which the rearrangement
      $\sum_na_{p(n)}$ no longer converges to the same limit.
\item \rri\ is defined like \rr\ except that $\sum_na_{p(n)}$ is
    required to diverge to $+\infty$ or to $-\infty$.
\end{ls}
\end{df}

Observe that the subseries numbers $\ss$ and $\ssi$ were defined in deliberate analogy with their corresponding rearrangement numbers, $\rr$ and $\rri$. The definition that seems to be missing, that of a rearrangement number $\rro$ analogous to $\sso$, was given in \cite{rearrangement}, but it was quickly proven that $\rr = \rro$, making the definition of $\rro$ redundant. The analogous equality does not seem to hold for the subseries numbers, so we will need to treat $\ss$ and $\sso$ separately. On the other hand, when considering permutations of terms rather than subseries as in the definitions of the rearrangement numbers, we had yet another option for how a permutation $p$ can reveal a series $\sum_{n \in \N}a_n$ to be conditionally convergent: that permuting terms makes the rearranged series $\sum_{n \in \N}a_{p(n)}$ converge to a different finite value. The version of the rearrangement number corresponding to this type of series disruption, namely $\rrf$, does not have a clear analog as a subseries number. 

If $\mathcal A$ is a family of subsets of $\N$ witnessing that every conditionally convergent series has a subseries  going to $\pm\infty$, then it automatically witnesses that every conditionally convergent series has a divergent subseries. This simple observation shows that $\ss \leq \ssi$, and a similar observation shows $\ss \leq \sso$:

\begin{thm}
$\ss \leq \ssi$ and $\ss \leq \sso$.
\end{thm}

Another simple observation is that all of our subseries numbers are at most $\ger c$, the cardinality of the continuum.

\begin{thm}
$\ssi \leq \ger c$ and $\sso \leq \ger c$. Consequently, $\ss \leq \ger c$ as well.
\end{thm}
\begin{pf}
The ``consequently'' part follows from the first part and the previous theorem.

Every conditionally convergent series has a subseries diverging to $\infty$ (for example, the sum of its positive terms). Thus $\mathcal P(\N)$ is a family of sets with the properties required in the definition of $\ssi$, and it follows that $\ssi \leq \ger c$.

Similarly, every conditionally convergent series has a subseries diverging by oscillation (such a subseries can be found by interleaving long stretches of negative terms with long stretches of positive terms). Thus $\mathcal P(\N)$ is a family of sets with the properties required in the definition of $\sso$, and it follows that $\sso \leq \ger c$.
\end{pf}

In the next section we will show that $\ss$, hence all three of the subseries numbers, is uncountable. Thus these three numbers qualify as cardinal characteristics of the continuum in the sense of \cite{hdbk}.

We end this section with a summary of the results that we will prove in the subsequent sections. Most of these results can be (and are) stated as inequalities comparing the subseries numbers with other cardinal characteristics of the continuum. The definitions of these other cardinals will be stated as needed later in the paper. We refer the reader to \cite{hdbk} for a thorough treatment of all the classical cardinal characteristics mentioned here (and others), and how they relate to one another. In what follows we prove:
\begin{multicols}{2}
\begin{ls}
\item $\ss \geq \ger s$
\item $\ss \geq \cov L$
\item $\ssi \geq \cov M$
\item $\sso \leq \non M$
\item[$\circ$] $\rr \leq \max \{ \ger b, \ss \}$
\item[$\circ$] $\rri \leq \max \{ \ger d, \ssi \}$
\item $\sso \leq \max \{ \ger b, \ss \}$
\end{ls}
\end{multicols}

\begin{center}
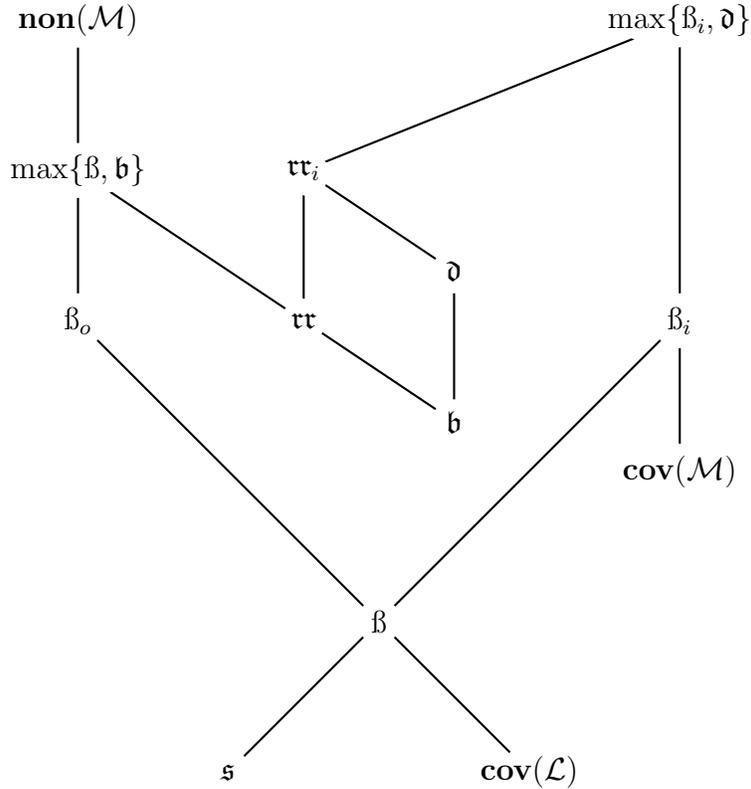
\begin{figure}
\begin{tikzpicture}[style=thick, xscale=2,yscale=2]

\draw (2,1) -- (3,2);
\draw (4,1) -- (3,2);
\draw (3,2) -- (1,4);
\draw (3,2) -- (5,4);
\draw (5,3) -- (5,4);
\draw (1,4) -- (1,5);
\draw (2.5,4) -- (2.5,5);
\draw (2.5,4) -- (1,5);
\draw (5,4) -- (5,6);
\draw (1,5) -- (1,6);
\draw (2.5,5) -- (5,6);

\draw (3.5,3.33) -- (2.5,4);
\draw (3.5,3.33) -- (3.5,4.33);
\draw (3.5,4.33) -- (2.5,5);

\draw [fill=white,white] (2,1) circle (4pt);  \node at (2,1) {$\ger s$};
\draw [fill=white,white] (4,1) circle (5pt);  \node at (4,1) {$\cov L$};
\draw [fill=white,white] (3,2) circle (4pt);  \node at (3,2) {$\ss$};
\draw [fill=white,white] (5,3) circle (5pt);  \node at (5,3) {$\cov M$};
\draw [fill=white,white] (1,4) circle (5pt);  \node at (1,4) {$\sso$};
\draw [fill=white,white] (5,4) circle (5pt);  \node at (5,4) {$\ssi$};
\draw [fill=white,white] (2.5,4) circle (4pt);  \node at (2.5,4) {$\rr$};
\draw [fill=white,white] (2.52,4.98) circle (4pt);  \node at (2.5,5) {$\rri$};
\draw [fill=white,white] (1,6) circle (5pt);  \node at (1,6) {$\non M$};
\draw [fill=white,white] (1,5) ellipse (10pt and 5pt);  \node at (1,5) {$\max \{\ss,\ger b\}$};
\draw [fill=white,white] (4.95,6) ellipse (10pt and 5pt);  \node at (5,6) {$\max \{\ssi,\ger d\}$};

\draw [fill=white,white] (3.5,3.33) circle (4pt);  \node at (3.5,3.33) {$\ger b$};
\draw [fill=white,white] (3.5,4.33) circle (4pt);  \node at (3.5,4.33) {$\ger d$};

\end{tikzpicture}
\caption{The subseries numbers compared to the rearrangement numbers and other small cardinals}
\end{figure}
\end{center}

Those inequalities we know to be consistently strict are marked with a filled-in circle. All the inequalities are summarized in visual form in a Hasse diagram in Figure 1. Not depicted in the diagram are the classical inequalities $\ger s \leq \ger d$ and $\cov M \leq \ger d$, which are proved in \cite{hdbk}, and the inequality $\cov L \leq \rr$, which was proved in \cite{rearrangement}.

By comparing these results with known facts about the random real model, the Cohen model, and the Mathias model, one may show easily that some of these inequalities are consistently strict. For example, in the random real model one has $\ger s = \aleph_1$ and $\cov L = \ger c$, so it follows that in this model we also have $\ger s < \ss$. These sorts of deductions will be made explicit in the relevant sections below. Of the four trivial inequalities mentioned above, namely 
\begin{multicols}{2}
\begin{ls}
\item $\ss \leq \ssi$
\item[$\circ$] $\ss \leq \sso$
\item[$\circ$] $\ssi \leq \ger c$
\item $\sso \leq \ger c$,
\end{ls}
\end{multicols}
\noindent we know that only two are not provably reversible (once again, it is the two marked with filled-in circles). In addition to these inequalities and the readily deduced consistency results that follow from them, we will also prove in Section~\ref{sec:Laver} that
\begin{ls}
\item consistently $\ss, \sso < \rr$.
\end{ls}
We prove this by showing this inequality holds in the Laver model. In fact, it is necessary only to prove that $\sso = \aleph_1$ in the Laver model, because $\ger b \leq \rr$, and it is well known that $\ger b = \ger c$ in the Laver model.

Our proof that $\sso=\aleph_1$ in the Laver model, presented in Section~\ref{sec:Laver}, is somewhat technical, and we expect that it will be accessible only to specialists who are intimately familiar with forcing arguments. The rest of the paper is intended to be accessible to a broader audience. No knowledge of forcing is required outside of Section~\ref{sec:Laver}. Some familiarity with Polish spaces and with classical cardinal characteristics of the continuum is likely to be helpful, but we do not assume the reader is an expert in these things.

\section{Padding with Zeros: $\ger s \leq \ss$} \label{pad1}

In this section, we obtain our first of two lower bounds for $\ss$ by showing that $\ger s \leq \ss$. It follows that $\ger s$ is a lower bound for all three of the subseries numbers. The main idea behind the proof is to begin with some conditionally convergent series, and then to produce a new conditionally convergent series by inserting a large number of zeros between consecutive terms of the original. This technique was introduced in \cite{rearrangement}, where it was used to show that $\ger b \leq \rr$. 

\begin{df}
Let $A$ and $B$ be infinite subsets of $\N$. $A$ is said to \emph{split} $B$ if both $B \cap A$ and $B \setminus A$ are infinite. The \emph{splitting number}, denoted $\ger s$, is the smallest cardinality of a family $\mathcal A$ of subsets of $\N$ such that every infinite set $B$ is split by some $A \in \mathcal A$.
\end{df}

\begin{thm}     \label{aboves:thm}
  $\ger s \leq \ss$.
\end{thm}

\begin{pf}
We must show that, given a set $\mathcal A$ of subsets of $\N$ with $|\mathcal A| < \ger s$, there is a conditionally convergent series $\sum_{n \in \N}a_n$ such that, for every $A \in \mathcal A$, the subseries $\sum_{n \in A}a_n$ remains convergent. To do this, we will begin with any conditionally convergent series $\sum_{n \in \N}b_n$, for example the alternating harmonic series $\sum_{n \in \N} (-1)^n/n$, and modify it by inserting a large number of zeroes between consecutive terms. 

If $\mathcal A$ is a family of subsets of $\N$ with $|\mathcal A| < \ger s$, then there is some infinite $B \subseteq \N$ that is not split by any $A \in \mathcal A$. Let $e_B: \N \rightarrow \N$ be the unique increasing function enumerating the elements of $B$.

If $\sum_{n \in \N}b_n$ is any conditionally convergent series, then we define a series $\sum_{n \in \N}a_n$ by setting
\begin{align*}
a_n =
\begin{cases}
b_k \ \  & \text{if } n = e_B(k) \\
0 \ \  & \text{if } n \notin B.
\end{cases}
\end{align*}
The series $\sum_{n \in \N}a_n$ has the same nonzero terms as $\sum_{n \in \N}b_n$, in the same order; the only difference is that many zeros have been inserted. In particular, $\sum_{n \in \N}a_n$ is conditionally convergent.

If $A \in \mathcal A$, then there are two possibilities: because $A$ does not split $B$, one of either $B \cap A$ or $B \setminus A$ is finite. If $B \cap A$ is finite, then $\sum_{n \in A}a_n$ has only finitely many nonzero terms, and it follows that this subseries is convergent. On the other hand, if $B \setminus A$ is finite then the nonzero terms of $\sum_{n \in A}a_n$ are exactly the same as the nonzero terms of $\sum_{n \in \N}b_n$, and in the same order, except that finitely many of these terms may have been deleted; it follows that $\sum_{n \in A}a_n$ is convergent. Thus, in either case, $\sum_{n \in A}a_n$ is convergent.
\end{pf}

\begin{cor}
All three of the subseries numbers are uncountable.
\end{cor}

\section{Randomly signed series: $\cov L \leq \ss$} \label{random} 

In this section we obtain our second lower bound for $\ss$ by showing that $\cov L \leq \ss$. The main idea behind the proof is to assign to each member of the measure space $2^\N$ a conditionally convergent series, and then to show that any given $A \subseteq \N$ gives rise to a divergent subseries of only a null set of these conditionally convergent series. This idea was used in \cite{rearrangement} to prove the corresponding inequality for the rearrangement number, $\cov L \leq \rr$. 

\begin{df}      \label{lebesgue:df}
Recall that $2^\N$, the set of all infinite sequences of zeros and ones, when equipped with the usual product topology, is known as the \emph{Cantor space}. The \emph{Lebesgue measure} on $2^\N$ is generated by declaring each basic open set of the form$\{s \in 2^\N : s(n) = i\}$ (where $n \in \N$ and $i \in \{0,1\}$) to have measure $1/2$. 
The \emph{covering number of the null ideal}, denoted $\cov L$, is the smallest cardinality of a collection $\mathcal A$ of subsets of $2^\N$ such that each $N \in \mathcal A$ has Lebesgue measure $0$, and $\bigcup \mathcal A = 2^\N$. In other words, $\cov L$ is the smallest number of null sets required to cover the Cantor space.
\end{df}

The value of \cov L would be unchanged if we used the real line equipped with its usual measure, or indeed any Polish space equipped with a continuous measure, in place of the Cantor space in the previous definition. 

We shall need a result of Rademacher~\cite{rade},
stated as a lemma below, about infinite series with randomly chosen
signs. 

\begin{la}[Rademacher]
Let $(c_n:n\in\bbb N)$ be any sequence of real numbers.  Let
$C \subseteq2^{\bbb N}$ be the set of all $s \in 2^\N$ for which $\sum_{n \in \N}(-1)^{s(n)}c_n$ converges.  Then the Lebesgue
measure of $C$ is 1 if $\sum_{n \in \N}{c_n}^2$ converges and 0 otherwise. 
\end{la}

In other words, if we attach signs randomly to the terms of the series
$\sum_nc_n$, the result will converge almost surely if $\sum_n{c_n}^2$
converges, and it will diverge almost surely otherwise.  

\begin{thm}     \label{abovecovN:thm}
$\cov L \leq \ss$.
\end{thm}

\begin{pf}
To begin, fix $A \subseteq \N$, and let us consider the question of whether we should expect a randomly signed harmonic series to converge or diverge on $A$. In other words, we would like to know the Lebesgue measure of the set
$$\mathrm{Div}_A = \textstyle \{s \in 2^\N : \sum_{n \in A}(-1)^{s(n)}/n \text{ diverges}\}.$$
Because the series $\sum_{n \in \N}1/n^2$ converges (and has only positive terms), the subseries $\sum_{n \in A}1/n^2$ converges too. Thus, by Rademacher's theorem, the Lebesgue measure of $2^\N \setminus \mathrm{Div}_A$ is $1$, so that the measure of $\mathrm{Div}_A$ is $0$.

Now consider any family $\mathcal A$ of fewer than \cov L subsets of $\N$. We will find a conditionally convergent series, of the form $\sum_{n \in \N}(-1)^{s(n)}/n$, such that, for any $A \in \mathcal A$, the subseries $\sum_{n \in A}(-1)^{s(n)}/n$ converges. This suffices to show that $\cov L \leq \ss$. 

Without loss of generality, we may assume $\N \in \mathcal A$. For each $A \in \mathcal A$, the set $\mathrm{Div}_A$ has Lebesgue measure $0$ in $2^\N$. It follows that there is some $s \in 2^\N$ that is not in $\mathrm{Div}_A$ for any $A \in \mathcal A$ (because otherwise we would have fewer than $\cov L$ null sets covering $2^\N$). This choice of $s$ guarantees that the series $\sum_{n \in \N}(-1)^{s(n)}/n$ is conditionally convergent (because $s \notin \mathrm{Div}_\N$), but that, for each $A \in \mathcal A$, the subseries $\sum_{n \in A}(-1)^{s(n)}/n$ converges.
\end{pf}

There is no provable inequality in either direction between \cov L and
\ger s.  Specifically, $\cov L < \ger s$ in the Mathias model and
$\ger s < \cov L$ in the random real model.  It follows that the lower bounds for
\ss\ in Theorems~\ref{aboves:thm} and \ref{abovecovN:thm} are
independent, and each of them can consistently be strict: $\ger s < \ss$ in the random real model, and $\cov L < \ss$ in the Mathias model.

\section{Generic sets: $\sso \leq \non M$} \label{cohen1} 

In this section we show that $\sso \leq \non M$. It follows that $\non M$ is also an upper bound for $\ss$, and that $\ss$ and $\sso$ are both consistently smaller than $\ger c$. Let us begin by recalling the definition of $\non M$:

\begin{df}      \label{baire:df} 
  A subset $M$ of a complete metric space $X$ is \emph{meager} (also
  called \emph{first category}) if it can be covered by countably many
  closed sets with empty interiors in $X$.  A \emph{comeager} set is
  the complement of a meager set; equivalently, it is a set that
  includes the intersection of countably many dense open subsets of
  $X$.  When $X$ is the Cantor space $2^{\bbb N}$, we denote the
  family of meager subsets of $2^{\bbb N}$ by \scr M.  The \emph{uniformity
    of Baire category}, denoted \non M, is the minimum cardinality of a
  non-meager subset of $2^{\bbb N}$.
\end{df}

The value of \non M would be unchanged if we used the real line, the Baire space $\N^\N$, or any other Polish space (provided it contains no isolated points) in
place of the Cantor space in the previous definition. 

\begin{thm}     \label{belownonM:thm}
$\sso \leq \non M$.
\end{thm}

\begin{pf}
Recall that we may identify $2^\N$ with $\mathcal P(\N)$, the power set of $\N$, via characteristic functions. In this way we may view $\mathcal P(\N)$ as a Polish space, and we may sensibly talk about meager, comeager, and non-meager sets of subsets of $\N$.

To prove the theorem, we will show that if $\sum_{n \in \N}a_n$ is a conditionally convergent series, then a ``generic'' subset $A$ of $\N$ gives rise to a subseries $\sum_{n \in A}a_n$ that diverges by oscillation. In other words, the set of all $A \subseteq \N$ that have this property is a comeager subset of $\mathcal P(\N)$. Consequently, if $\mathcal A$ is any non-meager subset of $\mathcal P(\N)$, then there is some $A \in \mathcal A$ such that $\sum_{n \in A}a_n$ diverges by oscillation. Since this is true for any conditionally convergent series, it follows that $\sso \leq \non M$.

Thus, to prove the theorem, it suffices to show that, for any conditionally convergent series $\sum_{n \in \N}a_n$,
$$\textstyle \{A \subseteq \N : \sum_{n \in A}a_n \text{ diverges by oscillation}\}$$
is a comeager subset of $\mathcal P(\N)$.

Fix a conditionally convergent series $\sum_{n \in \N}a_n$. Let $k \in \N$ and define
\[
U_k = \textstyle \{ A \subseteq \N : \sum_{n \in A \cap [1,m]} a_n \geq k \text{ for some } m\}
\]
We claim that for every $k \in \N$, $U_k$ is dense and open in $\mathcal P(\N)$. 

If $A \in U_k$, then there is
$m \geq k$ such that $\sum_{n \in A \cap [1,m]}a_n \geq k$, and any $A'$ that agrees
with $A$ on the interval $[1,m]$ is also in $U_k$. Recalling the definition of the product topology on $2^\N$, it follows that $U_k$ is open.

Next suppose that $A_0 \subseteq [1,\ell]$. Because $\sum_{n \in \N}a_n$ is conditionally convergent, if $P = \{n : a_n > 0\}$ then $\sum_{n \in P}a_n = \infty$. Letting $A = A_0 \cup (P \cap (\ell,\infty))$, we have $\sum_{n \in A}a_n = \infty$, which in particular implies $A \in U_k$. Again, recalling the definition of the product topology on $2^\N$, this shows that $U_k$ is dense in $2^\N$.

Similarly, for each $k \in \N$ define
\[
V_k = \textstyle \{ A \subseteq \N : \sum_{n \in A \cap [1,m]} a_n \leq -k \text{ for some } m\}
\]
By arguing in the same way as for $U_k$, we see that each $V_k$ is an open dense subset of $2^\N$.

The set $\mathcal O = \bigcap_{k \in \N}(U_k \cap V_k)$ is a countable intersection of dense open sets, and therefore is a comeager subset of $\mathcal P(\N)$. It is clear that if $A \in \mathcal O$ then $\sum_{n \in A}a_n$ diverges by oscillation.
\end{pf}

Our proof shows that for a given conditionally convergent series $\sum_{n \in \N}a_n$, there is a comeager set of $A \subseteq \N$ with the property that the subseries $\sum_{n \in A}a_n$ diverges by oscillation. In fact, we showed a bit more than this: there is a comeager set of $A \subseteq \N$ with the property that the subseries $\sum_{n \in A}a_n$ diverges by oscillation in the strongest possible way, namely
$$\limsup_{m \to \infty} \sum_{n \in A \cap [1,m]}a_n = \infty \qquad \text{and} \qquad \liminf_{m \to \infty} \sum_{n \in A \cap [1,m]}a_n = -\infty.$$

The inequality proved in this section can be strict. This follows from the main theorem of Section~\ref{sec:Laver} below, which states that $\sso = \aleph_1$ in the Laver model. It is well-known that $\non M = \aleph_2$ in that model, thus showing the consistency of $\sso < \non M$.

\section{More Polish spaces: $\ssi \geq \cov M$}

In this section we show that $\ssi \geq \cov M$. The idea of the proof is similar to that of Theorem~\ref{abovecovN:thm} in Section~\ref{random}, where we showed $\ss \geq \cov L$. We will begin by defining a Polish space $K$, and a way of associating to each $x \in K$ a conditionally convergent series $\sum_{n \in \N}a^x_n$. We will show that for any given $A \subseteq \N$, the set of all series for which $A$ defines a subseries going to $\infty$ determines a meager subset of $K$, and this will be used to conclude that $\ssi \geq \cov M$.

\begin{df}      \label{baire:df} 
  The \emph{covering number for Baire category}, denoted by $\cov M$, is the minimum cardinality of a family $\mathcal A$ of meager subsets of the Cantor space $2^\N$ with the property that $\bigcup \mathcal A = 2^\N$.
\end{df}

The value of \cov M would be unchanged if we used the real line, the Baire space $\N^\N$, or any other Polish space (provided it contains no isolated points) in
place of the Cantor space in the previous definition. For the proof below, the Polish space we will use is in fact homeomorphic to the Cantor space, but it will not be the usual representation as the product space $2^\N$.

\begin{thm}     \label{abovecovM:thm}
$\ssi \geq \cov M$.
\end{thm}

\begin{pf}

To begin, let us define a sequence of intervals as follows. Set $I_1 = [1,2]$, $I_2 = (2,6]$, $I_3 = (6,12]$, and in general let $I_k$ be the interval of length $2k$ that is adjacent (on the right) to $I_{k-1}$.

Let $D_k$ denote the set of all $k$-element subsets of $I_k$, and let us consider $D_k$ as a topological space by giving it the discrete topology (so up to homeomorphism, $D_k$ is simply the discrete space on $2k \choose k$ points). Let $K = \prod_{k \in \N}D_k$. As a set, $K$ consists of functions with domain $\N$ mapping each $k$ to some $k$-element subset of $I_k$. As a topological space, $K$ is homeomorphic to the Cantor space.

For each $x \in K$, we define a sequence $a^x_n$ as follows:
\begin{align*}
a^x_n = 
\begin{cases}
1/k^2 & \text{ if } n \in I_k \cap x(k) \\
-1/k^2 & \text{ if } n \in I_k \setminus x(k),
\end{cases}
\end{align*}
That is, on the interval $I_k$ our sequence $a^x_n$ will consist of $2k$ terms, $k$ of them equal to $1/k^2$ and $k$ of them equal to $-1/k^2$, with $x(k)$ telling us which terms are positive and which ones negative.

\vspace{3mm}

\noindent \emph{Claim:} $\sum_{n \in \N}a^x_n$ converges conditionally to $0$ for every $x \in K$.

\vspace{3mm}

\noindent \emph{Proof of claim.} Fix $x \in K$. For any given $k \in \N$, we have $\sum_{n \in I_k}a^x_n = 0$, because the positive and negative terms cancel. Now let $m \in \N$, and let us consider the partial sum $\sum_{n \leq m}a^x_n$. If $m \in I_k$, then
$$\sum_{n \leq m}a^x_n = \sum_{j < k}\sum_{n \in I_j}a^x_n + \sum_{n \in I_k \cap [1,m]}a^x_n = \sum_{n \in I_k \cap [1,m]}a^x_n.$$
Observe that $|\sum_{n \in I_k \cap [1,m]}a^x_n| \leq 1/k$, because $1/k$ is the total of all positive terms in $I_k$ and $-1/k$ is the total of all negative terms in $I_k$. Thus the partial sums $\sum_{n \leq m}a^x_n$ of the series $\sum_{n \in \N}a^x_n$ approach $0$ as $m$ grows large, which means that $\sum_{n \in \N}a^x_n = 0$. To see that the convergence is conditional, simply note that the sum of all the positive terms of the series is infinite:
$$\sum_{a^x_n > 0}a^x_n = \sum_{k \in \N} \left( \sum_{n \in I_k, a^x_n > 0}a^x_n \right) = \sum_{k \in \N}1/k = \infty.$$
This completes the proof of the claim. \hfill \qed \vspace{3mm}

Thus we have a Polish space $K$, and a way of associating a conditionally convergent series to every $x \in K$. For every $A \subseteq \N$, define
$$\mathrm{Inf}_A = \textstyle \{x \in K : \sum_{n \in A}a^x_n = \infty\}.$$

\vspace{3mm}

\noindent \emph{Claim:} For every $A \subseteq \N$, $\mathrm{Inf}_A$ is meager in $K$.

\vspace{3mm}

\noindent \emph{Proof of claim.} Notice that if $x,y \in K$ and if $x(k) = y(k)$ for all but perhaps finitely many $k \in \N$, then $x \in \mathrm{Inf}_A$ if and only if $y \in \mathrm{Inf}_A$. In other words, modifying a point of $K$ at finitely many coordinates cannot change whether it is in $\mathrm{Inf}_A$. This property is sometimes expressed by saying that $\mathrm{Inf}_A$ is a \emph{tail set} in $K$.

By the zero-one law for Baire Category (see Theorem 21.3 in \cite{Oxtoby}), every tail set in $K$ is either meager or co-meager. Thus, to prove the claim, it suffices to show that $\mathrm{Inf}_A$ is not co-meager for any $A \subseteq \N$.

Define a homeomorphism $h: K \to K$ by setting $h(x(k)) = I_k \setminus x(k)$ for all $k \in \N$. It is clear that $h$ is a homeomorphism from $K$ to itself. It is also clear that $a^{h(x)}_n = -a^x_n$ for all $n \in \N$ and all $x \in K$; in other words, $h$ has the effect of changing the sign of every term of the series $\sum_{n \in \N}a^x_n$. From this observation it follows that if $\sum_{n \in \N}a^x_n = \infty$, then $\sum_{n \in \N}a^{h(x)}_n = -\infty$. Thus $h$ maps $\mathrm{Inf}_A$ into its complement $K \setminus \mathrm{Inf}_A$.

If $\mathrm{Inf}_A$ were co-meager in $K$, then the image of $\mathrm{Inf}_A$ under $h$ would also be co-meager in $K$, because $h$ is a homeomorphism. But the intersection of two co-meager sets cannot be empty, so this would mean $\mathrm{Inf}_A \cap h[\mathrm{Inf}_A] \neq \emptyset$. This is not the case, by the previous paragraph. Thus $\mathrm{Inf}_A$ is not co-meager, finishing the proof of the claim.  \hfill \qed

\vspace{3mm}

To complete the proof of the theorem, consider any family $\mathcal A$ of fewer than \cov M subsets of $\N$. For each $A \in \mathcal A$, the set $\mathrm{Inf}_A$ is meager in $K$. It follows that there is some $x \in K$ that is not in $\mathrm{Inf}_A$ for any $A \in \mathcal A$ (because otherwise we would have fewer than $\cov M$ meager sets covering $K$). This choice of $x$ guarantees that it is not the case that $\sum_{n \in \N}a^x_n = \infty$ for any $A \in \mathcal A$. On the other hand, $\sum_{n \in \N}a^x_n$ converges conditionally to $0$, so we have found a conditionally convergent series, namely $\sum_{n \in \N}a^x_n$, such that, for any $A \in \mathcal A$, it is not the case that $\sum_{n \in A}a^x_n = \infty$. It follows that $\cov M \leq \ssi$. 
\end{pf}

Let us note that the corresponding result for rearrangement numbers, that $\rri \geq \cov M$, was proved in \cite{rearrangement}, though the proof there is fundamentally different (and easier); it hinges on the proof that $\rro = \rr$, and the analogue of this result does not seem to hold for the subseries numbers (although a version of it will be obtained in Section~\ref{ssovsss}).

In the random real model, $\cov M = \aleph_1$ while $\cov L = \ger c$. Using Theorem~\ref{abovecovN:thm}, $\cov L \leq \ss \leq \ssi$, and it follows that $\ssi > \cov M$ in the random real model. Thus the inequality proved in this section is consistently strict.

In the Cohen model, $\cov M = \ger c$ while $\non M = \aleph_1$. By Theorems~\ref{belownonM:thm} and \ref{abovecovM:thm}, it follows that $\sso < \ssi$ in the Cohen model. Thus the inequality $\ss \leq \ssi$ is consistently strict.

\section{Sparse sets: $\rr \leq \max \{\ss,\ger b\}$ and $\rri \leq \max \{\ssi,\ger d\}$} \label{ssvsrr} 

In this section we explore the relationship between the subseries numbers and their corresponding rearrangement numbers. We will prove that $\rr \leq \max \{\ss,\ger b\}$ and $\rri \leq \max \{\ssi,\ger d\}$. Very roughly, the main idea behind the proof is that the introduction of $\ger b$ and $\ger d$ allows us to take a divergent subseries and stretch its complement out onto a sufficiently sparse set, and in this way to build from it a rearranging permutation.

Let us begin by recalling the definitions of $\ger b$ and $\ger d$:

\begin{df}      \label{bd:df}
For functions $f,g:\bbb N\to\bbb N$, define $f\leq^*g$ to mean that
$f(n)\leq g(n)$ for all but finitely many $n\in\bbb N$.
\begin{ls}
\item The \emph{bounding number} \ger b is the minimum cardinality of a family
\scr B of functions $f:\bbb N\to\bbb N$ such that no single $g$ is
$\geq^*$ all $f\in\scr B$.
\item The \emph{dominating number} \ger d is the
minimum cardinality of a family \scr D of functions $f:\bbb N\to\bbb
N$ such that every $g:\bbb N\to\bbb N$ is $\leq^*$ at least one member
of \scr D.
\end{ls}
\end{df}

In the proof presented below, it is inconvenient to work directly with the definitions of $\ger b$ and $\ger d$. Instead we will use an alternative characterization in terms of families of ``sparse'' subsets of $\N$. This characterization is given in a slightly different form by Blass in Theorem 2.10 of \cite{hdbk}, and we refer the reader there for a proof. The earliest source for this characterization of $\ger b$ and $\ger d$ seems to be R. C. Solomon's \cite{sol}.

\begin{la}\label{lem:b&d}
Given $A, B \subseteq \N$, we say that $A$ is sparser than $B$ provided that, other than perhaps finitely often, there is not more than one element of $A$ between any two elements of $B$.
\begin{ls}
\item $\ger b$ is the minimum cardinality of a family $\mathcal B$ of subsets of $\N$ such that no single subset of $\N$ is sparser than every member of $\mathcal B$.
\item $\ger d$ is the minimum cardinality of a family $\mathcal D$ of subsets of $\N$ such that for any given subset of $\N$, some member of $\mathcal D$ is sparser.
\end{ls}
\end{la}

It will be convenient to define a specific type of permutation of the natural numbers, which we call a \emph{shuffle}. This definition is taken from Section 11 of \cite{rearrangement}, where shuffles are studied in a different context.

\begin{df}      \label{shuffle:df}
Let $A$ and $B$ be two infinite, coinfinite subsets of \bbb N.  The
\emph{shuffle} determined by $A$ and $B$ is the permutation $s_{A,B}$
of \bbb N that maps $A$ onto $B$ preserving order and maps
$\bbb N\setminus A$ onto $\bbb N\setminus B$ preserving order.  That is,
\[
s_{A,B}(n)=
\begin{cases}
  k\th\text{ element of }B&\text{ if }n\text{ is the }k\th\text{
    element of }A\\
  k\th\text{ element of }\bbb N\setminus B&\text{ if }n\text{ is the
  }k\th\text{ element of }\bbb N\setminus A
\end{cases}
\]
\end{df}

The usual proof of the Riemann rearrangement theorem makes use only of shuffles: the relative order of the positive terms remains unchanged, as does the relative order of the negative terms, and it is by splicing these two sets into each other in some way that one may rearrange a conditionally convergent series to diverge in any prescribed manner.

\begin{thm}     \label{ssvsrr:thm}
$\ $
\begin{enumerate}
\item $\rr \leq \max \{\ss,\ger b\}$. 
\item $\rri \leq \max \{\ssi,\ger d\}$. 
\end{enumerate}
\end{thm}

\begin{pf}
Both parts of this theorem are proved by variations of the same argument. We will therefore undertake to prove both parts simultaneously. Where necessary, we will break the argument into cases, whenever separate considerations are necessary for proving these two statements.

Let $\mathcal A$ be a family of subsets of $\N$ such that, for any conditionally convergent series $\sum_{n \in \N}a_n$, the subseries $\sum_{n \in A}a_n$ diverges for some $A \in \mathcal A$. For the proof of $(2)$, let us further assume that $\sum_{n \in A}a_n = \infty$ for some $A \in \mathcal A$.

Let $\mathcal B$ be a family of subsets of $\N$. For the proof of $(1)$, we will suppose that $\mathcal B$ satisfies the first part of Lemma~\ref{lem:b&d}: no single subset of $\N$ is sparser than every $B \in \mathcal B$. For the proof of $(2)$, we will suppose instead that $\mathcal B$ satisfies the second part of Lemma~\ref{lem:b&d}: for every subset of $\N$, some $B \in \mathcal B$ is sparser.

We will find a family $\mathcal C$ of permutations of $\N$ such that $|\mathcal C| \leq |\mathcal A| \cdot |\mathcal B|$ and, for every conditionally convergent series $\sum_{n \in \N}a_n$, the rearrangement $\sum_{n \in \N}a_{p(n)}$ diverges, and moreover, for the proof of $(2)$, we will show that in fact $\sum_{n \in \N}a_{p(n)} = \infty$. This suffices to prove the theorem.

We may (and do) also assume that every $A \in \mathcal A$ is neither finite nor co-finite. This assumption is without loss of generality, because no (co-)finite subset of $\N$ is any use for defining a divergent subseries of a conditionally convergent series; thus removing the (co-)finite subsets from $\mathcal A$ does not change whether or not it has the required properties. Similarly, we may (and do) assume that every $B \in \mathcal B$ is neither finite nor co-finite. Also, for reasons that become apparent later in the proof, let us assume (again, without loss of generality) that $\mathcal B$ is closed under the operation
$$\{b_1,b_2,b_3,\dots\} \mapsto \{b_1+1,b_2+2,b_3+3,\dots,b_i+i,\dots\},$$
where $b_1 < b_2 < b_3 < \dots < b_i < \dots$.

For each $A \in \mathcal A$ and $B \in \mathcal B$, define the permutation $p_{A,B}$ to be the inverse of the shuffle $s_{A,\N \setminus B}$. By our assumptions about the members of $\mathcal A$ and $\mathcal B$, this function is well-defined for every $A \in \mathcal A$ and $B \in \mathcal B$. The idea is that the rearrangement $\sum_{n \in \N}a_{p_{A,B}(n)}$ looks like the subseries $\sum_{n \in A}a_n$ written onto the set $\N \setminus B$, with the leftover terms, those indexed by members of $\N \setminus A$, written on the (very sparse) set $B$. Let
$$\mathcal C = \{p_{A,B} : A \in \mathcal A \text{ and } B \in \mathcal B\}.$$
It is clear that $|\mathcal C| \leq |\mathcal A| \cdot |\mathcal B|$, so it remains to show that, for every conditionally convergent series $\sum_{n \in \N}a_n$, the rearrangement $\sum_{n \in \N}a_{p(n)}$ diverges in the required way for some $p \in \mathcal C$.

Let $\sum_{n \in \N}a_n$ be a conditionally convergent series, and fix $A \in \mathcal A$ such that $\sum_{n \in A}a_n$ diverges; furthermore, for the proof of $(2)$, let us assume that the divergence is in the required manner, i.e. $\sum_{n \in A}a_n = \infty$. 

We will now define a sparse subset of $\N$ that is meant to capture the rate at which the subseries $\sum_{n \in A}a_n$ diverges. We consider three cases:
\begin{itemize}
\item If $\sum_{n \in A}a_n = \infty$, then the partial sums $\sum_{n \in A \cap [1,m]}a_n$ increase without bound. We may therefore find an increasing sequence $m_1,m_2,m_3,\dots$ of natural numbers such that
$$ \sum_{n \in A \cap (m_k,m_{k+1}]}a_n > 1$$
for all $k \in \N$, and furthermore
$$ \sum_{n \in A \cap (m_k,j]}a_n > 1$$
for all $j > m_{k+1}$. (This second condition is only used in the proof of $(2)$, and can be ignored for the proof of $(1)$.)
\item If $\sum_{n \in A}a_n = -\infty$, then the partial sums $\sum_{n \in A \cap [1,m]}a_n$ decrease without bound.  We may therefore find an increasing sequence $m_1,m_2,m_3,\dots$ of natural numbers such that
$$ \sum_{n \in A \cap (m_k,m_{k+1}]}a_n < -1$$
for all $k \in \N$.
\item If $\sum_{n \in A}a_n$ diverges by oscillation, then we may find some $c > 0$ such that the partial sums $\sum_{n \in A \cap [1,m]}a_n$ undergo infinitely many oscillations of size greater than $c$. More precisely, we may find an increasing sequence 
$m_1, m_1', m_2, m_2', m_3, m_3', \dots$
of natural numbers such that, for every $k \in \N$,
$$ \sum_{n \in A \cap (m_k,m_k']}a_n > c \qquad \text{and} \qquad \sum_{n \in A \cap (m_k',m_{k+1}]}a_n < -c.$$
\end{itemize}

Note that proving $(1)$ requires us to consider all three of these cases, as we do not have any information on the manner of divergence of the subseries $\sum_{n \in A}a_n$. Proving $(2)$ requires us to consider only the first case. 
Let $M^A = \{m_1^A,m_2^A,m_3^A,\dots\}$, where
$$m_k^A = |[1,m_k] \cap A|$$
for all $k \in \N$.

\vspace{3mm}

\noindent \emph{Claim 1:} Let $B \in \mathcal B$ and $A \in \mathcal A$, and let the maps $i \mapsto a_i$ and $i \mapsto b_i$ be the unique increasing enumerations of $A$ and $B$, respectively. If $n \in \N$ and $b_{\ell} < n < b_{\ell+1}$ for some $\ell \in \N$, then $p_{A,B}(n) = a_{n-\ell}$.

\vspace{3mm}

\noindent \emph{Proof of claim.} 
This follows immediately from the definitions.
\hfill \qed

\vspace{3mm}

\noindent \emph{Claim 2:} Let $A \in \mathcal A$, let $B_0 \in \mathcal B$, and let
$B = \{b_1+1,b_2+2,b_3+3,\dots\}$
where $i \mapsto b_i$ is the unique increasing enumeration of $B_0$. If the interval $(m_k^A,m_{k+1}^A]$ does not contain any members of $B_0$, then the terms $a_n$ with $n \in A \cap (m_k,m_{k+1}]$ will appear in order and consecutively in the rearranged series $\sum_{n \in \N}a_{p_{A,B}(n)}$.

\vspace{3mm}

\noindent \emph{Proof of claim.} 
Suppose $(m_k^A,m_{k+1}^A] \cap B_0 = \emptyset$ and let $\ell = |[1,m_k^A] \cap B_0|$, so that
$$b_\ell + \ell < m_k^A+\ell < m_{k+1}^A+\ell < b_{\ell+1}+\ell+1$$
where $b_i$ denotes the $i^{\mathrm{th}}$ element of $B_0$ as above. Let $e_A: \N \to \N$ denote the unique increasing enumeration of $A$. Applying Claim 1,
\begin{align*}
p_{A,B}(m_k^A+\ell+j) & = e_A(m_k^A+\ell+j-\ell) = e_A(m_k^A+j) \\
& = \text{the }j^{\mathrm{th}} \text{ member of }A \cap (m_k,m_{k+1}]
\end{align*}
for all $j \in [1,m^A_{k+1}-m_k^A]$.
\hfill \qed

\vspace{3mm}

To finish the proof of the theorem, we will consider two cases, according to whether $\mathcal B$ is assumed to satisfy the first part of Lemma~\ref{lem:b&d}, for the proof of $(1)$, or the second part, for the proof of $(2)$. 

\vspace{3mm}

\noindent \emph{Case 1: $M^A$ is not sparser than every $B \in \mathcal B$.}

\vspace{3mm}

Fix $B_0 \in \mathcal B$ such that $M^A$ is not sparser than $B_0$, and let
$$B = \{b_1+1,b_2+2,b_3+3,\dots\}$$
where $i \mapsto b_i$ is the unique increasing enumeration of $B_0$. Recall that $\mathcal B$ is closed under this transformation, so that $B \in \mathcal B$. We claim that the rearranged series $\sum_{n \in \N}a_{p_{A,B}(n)}$ diverges.

By the definition of ``sparser than'' there are infinitely many values of $k$ such that the interval $(m_k^A,m_{k+1}^A]$ does not contain any members of $B_0$. By claim 2, for each such interval the terms $a_n$ with $n \in A \cap (m_k,m_{k+1}]$ will appear in order and consecutively in the rearranged series $\sum_{n \in \N}a_{p_{A,B}(n)}$. 
If $\sum_{n \in A}a_n = \infty$, then this observation, together with our choice of the $m_k$, guarantees that the partial sums of the rearranged series $\sum_{n \in \N}a_{p_{A,B}(n)}$ will infinitely often increase by $1$, which implies that $\sum_{n \in \N}a_{p_{A,B}(n)}$ diverges. (Note: it does not follow that this series diverges to $\infty$. It is possible that while the partial sums of the rearranged series infinitely often increase by one, they always decrease later on in such a way that the rearrangement diverges by oscillation.) Similarly, if $\sum_{n \in A}a_n = -\infty$ then the partial sums of the rearranged series $\sum_{n \in \N}a_{p_{A,B}(n)}$ will infinitely often decrease by $1$, again implying that $\sum_{n \in \N}a_{p_{A,B}(n)}$ diverges. Lastly, if $\sum_{n \in A}a_n$ diverges by oscillation, then there is some $c > 0$ such that the partial sums of the rearranged series $\sum_{n \in \N}a_{p_{A,B}(n)}$ will infinitely often oscillate by $c$, once again implying that $\sum_{n \in \N}a_{p_{A,B}(n)}$ diverges. This completes the proof of $(1)$.

\vspace{3mm}

Before moving on to the second case, we will articulate one more claim. The proof is omitted, as it is nearly identical to the proof of Claim 2.

\vspace{3mm}

\noindent \emph{Claim 3:} Let $A \in \mathcal A$, let $B_0 \in \mathcal B$, and let
$B = \{b_1+1,b_2+2,b_3+3,\dots\}$
where $i \mapsto b_i$ is the unique increasing enumeration of $B_0$. If the interval $(m_k^A,m_{k+1}^A]$ contains exactly one member of $B_0$, say $b_\ell = m_k^A+j$, then
\[
p_{A,B}(m_k^A+\ell-1+i)=
\begin{cases}
  i\th \text{ element of } A \cap (m_k,m_{k+1}) \\ 
  \qquad\qquad\qquad\qquad \text{ if } i < j \\
  (i-1)^{\mathrm{st}} \text{ element of } A \cap (m_k,m_{k+1}) \\ 
  \qquad\qquad\qquad\qquad \text{ if } j < i \leq m^A_{k+1}+1-m^A_k
\end{cases}
\]
In other words, the terms $a_n$ with $n \in A \cap (m_k,m_{k+1}]$ will appear in the rearranged series $\sum_{n \in \N}a_{p_{A,B}(n)}$ consecutively, except that there is an extra term inserted.

\vspace{3mm}

\noindent \emph{Case 2: some $B \in \mathcal B$ is sparser than $M^A$.}

\vspace{3mm}

Fix $B_0 \in \mathcal B$ such that $B$ is sparser than $M^A$, and let
$$B = \{b_1+1,b_2+2,b_3+3,\dots\}$$
where $i \mapsto b_i$ is the unique increasing enumeration of $B_0$. Observe that $B \in \mathcal B$.
Because $(1)$ has already been proved and we are now aiming only at a proof of $(2)$, we will assume that $\sum_{n \in A}a_n = \infty$. We claim that $\sum_{n \in \N}a_{p_{A,B}(n)} = \infty$ as well.

For all but finitely many values of $k$, the interval $(m_k^A,m_{k+1}^A]$ contains at most one point of $B_0$. Together with claims 2 and 3, this implies that we may partition $\N$ into intervals $I_1,I_2,I_3,\dots$ such that for all but finitely many values of $k$, $p_{A,B}$ maps $I_k$ to the set $A \cap (m_k,m_{k+1}]$ in an order-preserving fashion, with perhaps one exception, namely a single $j \in I_k$ mapping to some $n \notin A$. For all but finitely many $n$ we have $|a_n| < \nicefrac{1}{2}$; thus, with finitely many exceptions, if $j \in I_k$ with $p_{A,B}(j) \notin A$, then $|a_{p_{A,B}}(j)| < \nicefrac{1}{2}$.

Thus $\N$ can be divided into intervals $I_k$, and on all but finitely many of these intervals we have either
\begin{enumerate}[(i)]
\item all partial sums of $\sum_{n \in I_k}a_{p_{A,B}(n)}$ and $\sum_{n \in A \cap (m_k,m_{k+1}]}a_n$ are identical (this happens if $(m^A_k,m^A_{k+1}]$ contains no members of $B_0$), or
\item the partial sums of $\sum_{n \in I_k}a_{p_{A,B}(n)}$ and $\sum_{n \in A \cap (m_k,m_{k+1}]}a_n$ differ by less than $\nicefrac{1}{2}$ (this happens if $(m^A_k,m^A_{k+1}]$ contains one member of $B_0$).
\end{enumerate}

Together with our choice of the $m_k$, this is enough to conclude that $\sum_{n \in \N}a_{p_{A,B}(n)} = \infty$, completing the proof of $(2)$.
\end{pf}

We do not know whether either of the inequalities proved in this section can be strict. It was proved in \cite{rearrangement} that $\ger b \leq \rr$, so the consistency of $\rr < \max\{\ss,\ger b\}$ would imply the consistency of $\rr < \ss$. Even this latter, ostensibly easier problem remains open:

\begin{que}
Is $\rr < \ss$ consistent?
\end{que}

The twin question of whether $\ss < \rr$ is consistent will be answered affirmatively in Section~\ref{sec:Laver}.

\section{More sparse sets: $\sso \leq \max \{\ss,\ger b\}$} \label{ssovsss} 

In this section we present another argument involving sparse sets, akin to the proof in the previous section. This time we will prove a nontrivial relationship between two of the subseries numbers, $\ss$ and $\sso$.

\begin{thm}\label{thm:ssovsss}
$\sso \leq \max \{\ss,\ger b\}$.
\end{thm}

\begin{la}\label{up/down}
Let $\sum_{n \in \N}a_n$ be a conditionally convergent series, and let $A \subseteq \N$. If $\sum_{n \in A}a_n = \infty$ then $\sum_{n \notin A}a_n = -\infty$, and if $\sum_{n \in A}a_n = -\infty$ then $\sum_{n \notin A}a_n = \infty$.
\end{la}
\begin{pf}
Suppose $\sum_{n \in \N}a_n = c$. Given any $M > 0$, for large enough $k$ we have $\sum_{n \in A \cap [1,k]}a_n > M$ and $\sum_{n \in [1,k]}a_n$ within $1$ of $c$. Consequently, for large enough $k$ we have $\sum_{n \in [1,k] \setminus A}a_n < c-M+1$. This shows that $\sum_{n \in [1,k] \setminus A}a_n$ decreases without bound whenever $\sum_{n \in A \cap [1,k]}a_n$ increases without bound. Thus $\sum_{n \notin A}a_n = -\infty$ whenever $\sum_{n \in A}a_n = \infty$. A similar argument shows $\sum_{n \notin A}a_n = \infty$ whenever $\sum_{n \in A}a_n = -\infty$.
\end{pf}

This lemma shows that if we know a set $A$ on which our conditionally convergent sum goes to $\infty$, then we know a set, namely $\N \setminus A$, on which it goes to $-\infty$. This simple observation will be crucial to the proof of Theorem~\ref{thm:ssovsss}.

\begin{pf}[Proof of Theorem~\ref{thm:ssovsss}]
Let $\mathcal A$ be a family of subsets of $\N$ such that, for any conditionally convergent series $\sum_{n \in \N}a_n$ the subseries $\sum_{n \in A}a_n$ diverges for some $A \in \mathcal A$. Let $\mathcal B$ be a family of subsets of $\N$ such that no single subset of $\N$ is sparser than every $B \in \mathcal B$. Finally, let $\mathcal S$ be a family of subsets of $\N$ such that every infinite subset of $\N$ is split by some member of $\mathcal S$.

We will find a family $\mathcal C$ of subsets of $\N$ such that $|\mathcal C| \leq |\mathcal A| \cdot |\mathcal B| \cdot |\mathcal S|$ and, for any conditionally convergent series $\sum_{n \in \N}a_n$, the subseries $\sum_{n \in A}a_n$ diverges by oscillation for some $A \in \mathcal C$. This proves that $\sso \leq \max \{\ss,\ger b,\ger s\}$. As $\ger s \leq \ss$ by Theorem~\ref{aboves:thm}, this suffices to prove the theorem.

We may (and do) assume without loss of generality that if $A \in \mathcal A$ then $\N \setminus A \in \mathcal A$.

Let $A \in \mathcal A$, $B \in \mathcal B$, and $S \in \mathcal S$. Let $b_1,b_2,b_3,\dots$ denote the elements of $B$ in increasing order, and define
$$C_{A,B,S} = \left( \bigcup_{n \in S} A \cap (b_n,b_{n+1}] \right) \cup \left( \bigcup_{n \notin S} (b_n,b_{n+1}] \setminus A \right).$$
Let
$$\mathcal C = \mathcal A \cup \{C_{A,B,S} : A \in \mathcal A, B \in \mathcal B, \text{ and }S \in \mathcal S\}.$$
It is clear that $|\mathcal C| \leq |\mathcal A| \cdot |\mathcal B| \cdot |\mathcal S|$, so it remains to prove that for any conditionally convergent series $\sum_{n \in \N}a_n$, the subseries $\sum_{n \in C}a_n$ diverges by oscillation for some $C \in \mathcal C$.

Fix a conditionally convergent series $\sum_{n \in \N}a_n$. If there is some $A \in \mathcal A$ such that the subseries $\sum_{n \in A}a_n$ diverges by oscillation, then we are done, because $A \in \mathcal C$. Thus let us suppose that there is some $A \in \mathcal A$ such that either $\sum_{n \in A}a_n = \infty$ or $\sum_{n \in A}a_n = -\infty$. By replacing $A$ with $\N \setminus A$ if necessary (recall that $\mathcal A$ is closed under complementation), Lemma~\ref{up/down} shows that we may assume $\sum_{n \in A}a_n = \infty$.

Because the partial sums $\sum_{n \in A \cap [1,m]}a_n$ increase without bound, we may use recursion to find an increasing sequence $m_1,m_2,m_3,\dots$ of natural numbers such that
$$ \sum_{n \in A \cap (m_k,m_{k+1}]}a_n > \sum_{n \in [1,m_k]}|a_n|+1$$
for all $k \in \N$. Fix $B \in \mathcal B$ such that $M = \{m_1,m_2,m_3,\dots\}$ is not sparser than $B$.

Let $b_1,b_2,b_3,\dots$ denote the elements of $B$ in increasing order, and define
$$X = \{n \in \N : (b_n,b_{n+1}] \cap M \geq 2\}.$$
Because $M$ is not sparser than $B$, $X$ is infinite. Fix $S \in \mathcal S$ such that both $S \cap X$ and $X \setminus S$ are infinite.

We claim that $\sum_{n \in C_{A,B,S}}a_n$ diverges by oscillation. To see this, we will show that the values of the partial sums $\sum_{n \in C_{A,B,S} \cap [1,m]}a_n$ are infinitely often greater than $1$ and infinitely often less than $0$.

Let $\ell \in S \cap X$, and let $m_k, m_{k+1}$ denote the first two members of $M$ contained in the interval $(b_\ell,b_{\ell+1}]$. By the definition of $C_{A,B,S}$, we have $C_{A,B,S} \cap (m_k,m_{k+1}] = A \cap (m_k,m_{k+1}]$. By this observation, and by our choice of the $m_k$, we have
\begin{align*}
 \sum_{n \in C_{A,B,S} \cap [1,m_{k+1}]}a_n
& =  \sum_{n \in C_{A,B,S} \cap [1,m_k]}a_n +\sum_{n \in C_{A,B,S} \cap (m_k,m_{k+1}]}a_n \\
& \geq  \sum_{n \in C_{A,B,S} \cap [1,m_k]}-|a_n| +\sum_{n \in C_{A,B,S} \cap (m_k,m_{k+1}]}a_n \\
& =  -\sum_{n \in C_{A,B,S} \cap [1,m_k]}|a_n| +\sum_{n \in A \cap (m_k,m_{k+1}]}a_n \\
& > 1.
\end{align*}
Thus infinitely many of the partial sums of the subseries $\sum_{n \in C_{A,B,S}}a_n$ are greater than $1$.

To finish the proof, first pick $N$ large enough that for all $\ell \geq N$, if $m \geq b_\ell$ and $m' > n$, then $\sum_{n \in (m,m']}a_n < 1$. This is possible because $\sum_{n \in \N}a_n$ converges.

Let $\ell \in X \setminus S$ with $\ell \geq N$, and let $m_k, m_{k+1}$ denote the first two members of $M$ contained in the interval $(b_\ell,b_{\ell+1}]$. By the definition of $C_{A,B,S}$, we have $C_{A,B,S} \cap (m_k,m_{k+1}] = (m_k,m_{k+1}] \setminus A$, and by our choice of $N$ we have 
$$\textstyle \sum_{n \in (m_k,m_{k+1}] \setminus A}a_n < 1-\sum_{n \in A \cap (m_k,m_{k+1}]}a_n$$
Applying these observations, we have
\begin{align*}
 \sum_{n \in C_{A,B,S} \cap [1,m_{k+1}]}a_n
& =  \sum_{n \in C_{A,B,S} \cap [1,m_k]}a_n + \sum_{n \in C_{A,B,S} \cap (m_k,m_{k+1}]}a_n \\
& \leq  \sum_{n \in C_{A,B,S} \cap [1,m_k]}|a_n| + \sum_{n \in C_{A,B,S} \cap (m_k,m_{k+1}]}a_n \\
& =  \sum_{n \in C_{A,B,S} \cap [1,m_k]}|a_n| + \sum_{n \in (m_k,m_{k+1}] \setminus A}a_n \\
& <  \sum_{n \in C_{A,B,S} \cap [1,m_k]}|a_n| + \left( 1-\sum_{n \in A \cap (m_k,m_{k+1}]}a_n \right) \\
& < 0.
\end{align*}
Thus infinitely many of the partial sums of the subseries $\sum_{n \in C_{A,B,S}}a_n$ are less than $0$. This completes the proof that $\sum_{n \in C_{A,B,S}}a_n$ diverges by oscillation, which in turn completes the proof of the theorem.
\end{pf}

In the random real model, $\ger b = \aleph_1$ and $\cov L = \ger c$. By Theorem~\ref{abovecovN:thm}, this shows that it is consistent to have $\ger b < \ss$. In Section~\ref{sec:Laver} we will prove the consistency of $\ss < \ger b$. Thus there is no provable inequality between $\ss$ and $\ger b$, which shows that $\max \{ \ss, \ger b \}$ cannot simply be replaced with either $\ss$ or $\ger b$ in the statement of the previous theorem. Furthermore, the results in Section~\ref{sec:Laver} show that
$$\aleph_1 = \sso < \max \{\ss,\ger b\} = \aleph_2$$
in the Laver model, so that the inequality proved in this section is consistently strict.

\section{The Laver model: $\ss = \sso < \ger b = \rr$}\label{sec:Laver}

In this section we will prove that $\sso = \aleph_1$ in the Laver model. It is well-known that $\ger b = \aleph_2 = \ger c$ in the Laver model, so, combined with the inequalities $\ger b \leq \rr$ (which was proved in \cite{rearrangement}) and $\ss \leq \sso$, this result shows the consistency of $\ss < \rr$.

The idea of the proof is that we will define an intermediate cardinal characteristic, which we call the \emph{almost splitting number} and denote $\salmost$, then prove that $\sso \leq \salmost$ (always, not just in the Laver model), and finally prove that $\salmost = \aleph_1$ in the Laver model.

\begin{df}
Let $\bar I = \langle I_k : k \in \N \rangle$ be a sequence of finite subsets of $\N$ with $\max (I_n) < \min (I_{n+1})$ for every $n$. (Usually, in what follows, $\bar I$ will be a partition of $\N$ into finite intervals.) For each $k$, let $B_k \subseteq I_k$, and denote $\bar B = \langle B_k : k \in \N \rangle$. Let $\bar a = \langle a_n : n \in \N \rangle$ denote a sequence of real numbers, with $0 \leq a_n \leq 1$ for every $n$. If $r,s \geq 0$, we say that $(\bar I, \bar B, \bar a)$ is an $(r,s)$\emph{-sequence} provided that 
$$\lim_{k \to \infty} \sum_{n \in B_k}a_n \,=\, r \qquad \text{and} \qquad \lim_{k \to \infty} \sum_{n \in I_k \setminus B_k}a_n \,=\, s.$$
We say that an infinite set $D \subseteq \N$ \emph{almost splits} $(\bar I, \bar B, \bar a)$ if there is an infinite set $E \subseteq \N$ such that 
$$\lim_{k \in E} \sum_{n \in D \cap B_k}a_n \,=\, r \qquad \text{and} \qquad \lim_{k \in E} \sum_{n \in D \cap I_k \setminus B_k}a_n \,=\, 0.$$
We say that an infinite set $D \subseteq \N$ \emph{totally splits} $(\bar I, \bar B, \bar a)$ if there is an infinite set $E \subseteq \N$ such that for all $k \in E$ we have $D \cap I_k = B_k$.
\begin{itemize}
\item The \emph{almost splitting number}, denoted $\salmost$, is the least cardinality of a family $\mathcal D$ of subsets of $\N$ such that, for every countable family $(\bar I^m, \bar B^m, \bar a^m)$ of $(r^m,s^m)$-sequences, $m \in \N$, there is some $D \in \mathcal D$ almost splitting each one of them.
\item The \emph{total splitting number}, denoted $\stotal$, is defined like $\salmost$, except that we require $D$ to totally split the $(r^m,s^m)$-sequence $(\bar I^m, \bar B^m, \bar a^m)$ for every $m \in \N$.
\end{itemize}
\end{df}

If a set $D$ totally splits an $(r,s)$-sequence, then it also almost splits the sequence. It follows that $\salmost \leq \stotal$. 

The cardinal characteristic $\salmost$, though employed only as a supporting actor in our proof below, may have some independent interest. Indeed, $\salmost$ is closely related to the \emph{finitely splitting number} $\ger f \ger s$ that was defined by Kamburelis and W\textpolhook{e}glorz in \cite{K&W}. They proved that $\ger f \ger s = \max \{ \ger b, \ger s \}$. Additionally, it is not hard to show 
$$\max \{ \ger b, \ger s \} \,=\, \ger f \ger s \,\leq\, \stotal \,\leq\, \non M.$$
The first inequality is a direct consequence of the definitions, and the second inequality is proved by an argument similar to that in Theorem~\ref{belownonM:thm} above.

Thus it would seem that slight alterations in the definition of $\salmost$ result in cardinals $\ger f \ger s$ and $\stotal$ that are provably $\geq \ger b$. Despite this, we will show that $\salmost < \ger b$ in the Laver model.

\begin{thm}
$\sso \leq \salmost$.
\end{thm}
\begin{proof}
Let $\mathcal D$ be a family of infinite subsets of $\N$ that satisfies the definition of $\salmost$. In fact, we will not need the full force of the definition, but may content ourselves with the following fact: for every $r,s \geq 1$, every $(r,s)$-sequence is almost split by some $D \in \mathcal D$. We will show that $\mathcal D$ also satisfies the definition of $\sso$: for every conditionally convergent series $\sum_{n \in \N} a_n$ of real numbers, there is some $D \in \mathcal D$ such that the subseries $\sum_{n \in D} a_n$ diverges by oscillation.

Let $\sum_{n \in \N}a_n$ be a conditionally convergent series of real numbers. Without loss of generality, we may assume that $-1 \leq a_n \leq 1$ for every $n$. Let
$$P = \set{n \in \N}{a_n \geq 0} \qquad \text{and} \qquad N = \set{n \in \N}{a_n < 0}.$$
Using recursion, we will now define a partition of the natural numbers into finite intervals. To begin, let $J_1$ be the largest interval in $\N$ containing $1$ and having the property that $\sum_{n \in J_1}|a_n| \leq 5$ (this interval is finite because $\sum_{n \in \N}|a_n| = \infty$). Supposing $J_1, J_2, \dots, J_{k-1}$ are already defined, choose $J_k$ to be the largest interval in $\N$ of the form $[\max J_{k-1}+1,\ell]$ and having the property that $\sum_{n \in J_k}|a_n| \leq 5^k$. This recursive construction defines a sequence $\langle J_k : k \in \N \rangle$ of consecutive intervals of natural numbers with the following useful property:

\vspace{3mm}

\noindent \emph{Claim:}
$\displaystyle \lim_{k \to \infty} \left( \frac{5^k}{2} - \sum_{n \in J_k \cap P}a_n \right) = 0 \ $  and 
$\displaystyle \ \lim_{k \to \infty} \left( \frac{5^k}{2} + \sum_{n \in J_k \cap N}a_n \right) = 0$.

\vspace{3mm}

\noindent \emph{Proof of claim.} 
Our construction of the $J_k$ implies 
$$\lim_{k \to \infty} \left( 5^k - \sum_{n \in J_k}|a_n| \right) = \lim_{k \to \infty} |a_{\max J_k +1}| = 0.$$
Because $\sum_{n \in \N}a_n$ converges conditionally,
$$\lim_{k \to \infty} \left( \sum_{n \in J_k \cap P}a_n + \sum_{n \in J_k \cap N}a_n \right) = 0.$$
The claim is readily deduced from these observations. \hfil \qed \vspace{3mm}

In other words, for large enough $k$ we have $\sum_{n \in P \cap J_k}a_n \approx \nicefrac{5^k}{2}$ and $\sum_{n \in N \cap J_k}a_n \approx \nicefrac{-5^k}{2}$.

We are now ready to define the $(r,s)$-sequence to which we will apply our family $\mathcal D$. Let
\begin{itemize}
\item $I_k = J_{2k-1} \cup J_{2k}$ for every $k \in \N$,
\item $B_k = (P \cap J_{2k-1}) \cup (N \cap J_{2k})$ for every $k \in \N$, and
\item $x_n = \nicefrac{|a_n|}{5^k}$ for every $n \in \N$, where $k$ is taken to be the unique natural number with $n \in J_k$.
\end{itemize}
Let $\bar I = \langle I_k : k \in \N \rangle$, $\bar B = \langle B_k : k \in \N \rangle$, and $\bar x = \langle x_n : n \in \N \rangle$, and observe that $(\bar I, \bar B, \bar x)$ is a $(1,1)$-sequence.

By our assumptions about the family $\mathcal D$, there is some $D \in \mathcal D$ that almost splits $(\bar I, \bar B, \bar x)$. We will show that the sum $\sum_{n \in D}a_n$ diverges by oscillation to complete the proof of the theorem.

Fix an infinite set $E \subseteq \N$ such that 
$$\lim_{k \in E} \sum_{n \in D \cap B_k}x_n \,=\, 1 \qquad \text{and} \qquad \lim_{k \in E} \sum_{n \in D \cap I_k \setminus B_k}x_n \,=\, 0.$$

From our claim and the definition of the $x_k$ and $B_k$, it follows that
$$\lim_{k \to \infty} \sum_{n \in J_{2k-1} \cap B_k}x_k \,=\, \lim_{k \to \infty} \sum_{n \in J_{2k} \cap B_k}x_k \,=\, \frac{1}{2}.$$
Combining this with the equation $\lim_{k \in E} \sum_{n \in D \cap B_k}x_n = 1$ and the fact that $B_k \subseteq I_k = J_{2k-1} \cup J_{2k}$, we see that
$$\lim_{k \to \infty} \sum_{n \in D \cap B_k \cap J_{2k-1}}x_n \,=\, \lim_{k \to \infty} \sum_{n \in D \cap B_k \cap J_{2k}}x_n \,=\, \frac{1}{2}.$$
In particular,
\begin{equation}
\sum_{n \in D \cap B_k \cap J_{2k-1}}x_n \,=\, \sum_{n \in D \cap P \cap J_{2k-1}}x_n \,>\, \frac{3}{8} 
\end{equation}
\begin{equation}
\sum_{n \in D \cap B_k \cap J_{2k}}x_n \,=\, \sum_{n \in D \cap N \cap J_{2k}}x_n \,>\, \frac{3}{8} 
\end{equation}
for large enough values of $k$, when $k \in E$. Similarly, we also have
\begin{equation}
\sum_{n \in (D \setminus B_k) \cap J_{2k-1}}x_n \,=\, \sum_{n \in D \cap N \cap J_{2k-1}}x_n \,<\, \frac{1}{8} 
\end{equation}
\begin{equation}
\sum_{n \in (D \setminus B_k) \cap J_{2k}}x_n \,=\, \sum_{n \in D \cap P \cap J_{2k}}x_n \,<\, \frac{1}{8} 
\end{equation}
for sufficiently large $k \in E$. Using our definition of the $x_k$, we may multiply both sides of $(1)$ and $(3)$ by $\pm 5^{2k-1}$ and both sides of $(2)$ and $(4)$ by $\pm 5^{2k}$ to obtain
\begin{equation}\tag{$1'$}
\sum_{n \in D \cap P \cap J_{2k-1}}a_n \,>\, \frac{3}{8} 5^{2k-1}
\end{equation}
\begin{equation}\tag{$2'$}
\sum_{n \in D \cap N \cap J_{2k}}a_n \,<\, -\frac{3}{8} 5^{2k}
\end{equation}
\begin{equation}\tag{$3'$}
\sum_{n \in D \cap N \cap J_{2k-1}}a_n \,>\, -\frac{1}{8} 5^{2k-1}
\end{equation}
\begin{equation}\tag{$4'$}
\sum_{n \in D \cap P \cap J_{2k}}a_n \,<\, \frac{1}{8} 5^{2k}
\end{equation}
for sufficiently large $k \in E$.
Combining $(1')$ with $(3')$ and $(2')$ with $(4')$, we obtain
$$\sum_{n \in D \cap J_{2k-1}}a_n \,>\, \frac{1}{4} 5^{2k-1}$$
$$\sum_{n \in D \cap J_{2k}}a_n \,<\, -\frac{1}{4} 5^{2k}$$
for sufficiently large $k \in E$.

Thus the partial sums of $\sum_{n \in D}a_n$ increase by at least $\nicefrac{5^{2k-1}}{4}$ and then decrease by at least $\nicefrac{5^{2k}}{4}$ on the interval $I_k$, for infinitely many values of $k$, namely all sufficiently large $k \in E$. Furthermore,
$$\sum_{n < \min I_k}|a_n| \,=\, \sum_{\ell \leq 2k-2} \left( \sum_{n \in J_\ell} |a_n| \right) \,<\, \sum_{\ell \leq 2k-2} 5^\ell \,=\, \frac{1}{4}\left( 5^{2k-1}-1 \right)$$
for all $k \in \N$.
Thus, for all sufficiently large $k \in E$, we have
\begin{align*}
\sum_{n \in D,\, n \leq \max J_{2k-1}}a_n \,\geq\, & \left( -\sum_{n < \min I_k}|a_n| \right) + \left( \sum_{n \in D \cap J_{2k-1}}a_n \right) \\
>\, & -\frac{1}{4}\left( 5^{2k-1}-1 \right) + \frac{1}{4} 5^{2k-1} \,=\, \frac{1}{4} \ \ \ \ \text{and} \\ & \\
\sum_{n \in D,\, n \leq \max J_{2k}}a_n \,\leq\, & \left( \sum_{n < \min J_{2k}}|a_n| \right) + \left( \sum_{n \in D \cap J_{2k}}a_n \right) \\
<\, & \frac{1}{4}\left( 5^{2k}-1 \right) - \frac{1}{4} 5^{2k} \,=\, -\frac{1}{4}.
\end{align*}

Thus we see that the partial sums of the series $\sum_{n \in D}a_n$ are infinitely often greater than $\nicefrac{1}{4}$ and infinitely often less than $-\nicefrac{1}{4}$. It follows that $\sum_{n \in D}a_n$ diverges by oscillation.
\end{proof}

In the remainder of this section we assume the reader is familiar with the method of forcing. $\mathbb L$ denotes the \emph{Laver forcing}, which is the set of all Laver trees, ordered by inclusion. A \emph{Laver tree} is a set $T$ of finite sequences of natural numbers such that
\begin{itemize}
\item $T$ is a \emph{tree}, which means that $T$ contains all initial segments of any member of $T$,
\item $T$ has a \emph{stem}, which is a sequence $s \in T$ with the property that every other member of $T$ either extends $s$ or is an initial segment of $s$, and
\item every member of $T$ extending the stem $s$ has infinitely many immediate successors.
\end{itemize}
Because $\mathbb L$ is ordered by inclusion, stronger conditions are trees containing fewer sequences. This notion of forcing is proper. The \emph{Laver model} refers to any model obtained by an $\omega_2$-stage countable-support iteration of the Laver forcing over a model of GCH.

We fix some notation for Laver forcing $\LL$. For $S,T \in \LL$ we write $S \leq_0 T$ if $S \leq T$ and $\stem (S) = \stem (T)$.
For $\sigma \in T$ with $\stem (T) \sub \sigma$, $T_\sigma = \{ \tau \in T : \tau \sub \sigma$ or $\sigma \sub \tau \}$ is
the {\em subtree of $T$ given by $\sigma$}. Clearly $\stem (T_\sigma) = \sigma$. Next, $\succ_T (\sigma) = \{ n
\in \omega : \sigma \ha n \in T \}$ is the {\em successor level} of $\sigma$ in $T$. We say that $F \sub T$ is a 
{\em front} if for every $x \in [T]$ there is a unique $\sigma \in F$ with $\sigma \sub x$. A front is in particular
a maximal antichain in $T$ (but a maximal antichain need not be a front). 

\begin{mainlem}
Assume $T \in \LL$, $r, s \geq 0$, and $\epsilon > 0$ are reals, 
and $\dot I$, $\dot B$, $\dot C$, $(\dot a_k : k \in \dot I)$, $\dot b$, and $\dot c$ are $\LL$-names such
that $T$ forces
\begin{itemize}
\item $\dot I \sub \omega$ is finite, $\dot B \sub \dot I$, $\dot C = \dot I \sem \dot B$,
\item all $\dot a_k$ are reals between $0$ and $1$, $\sum_{k \in \dot B} \dot a_k = \dot b$, $\sum_{k \in \dot C} \dot a_k = \dot c$,
\item $| \dot b - r |, |\dot c - s | < \epsilon$.
\end{itemize}
Also assume that no subtree of $T$ with the same stem decides the value of $\min (\dot I)$. Then there are 
$S \leq_0 T$, a front $F \sub S$, sequences $\bar I_\sigma = (I_{\sigma,n} : n \in \succ_S (\sigma))$,
$\bar B_\sigma = (B_{\sigma,n} : n \in \succ_S (\sigma))$, and $\bar a_\sigma = (a_{\sigma,k} : k \in \omega )$,
and reals $r_\sigma, s_\sigma \geq 0$ for all $\sigma \in S$ with $\stem (S) \sub \sigma$ and $\sigma \subsetneq \tau$
for some $\tau \in F$ such that all $(\bar I_\sigma, \bar B_\sigma, \bar a_\sigma)$ are $(r_\sigma, s_\sigma)$-sequences
and such that whenever $D \in \omoms$ almost splits all $(\bar I_\sigma, \bar B_\sigma, \bar a_\sigma)$, then there is
$S' \leq_0 S$ such that $S'_\sigma = S_\sigma$ for all $\sigma \in F \cap S'$ and
\begin{itemize}
\item $S' \forces | \sum_{k \in \dot B \cap D} \dot a_k - r | < 2 \epsilon$ and $ \sum_{k \in \dot C \cap D} \dot a_k < 2 \epsilon$.
\end{itemize}
\end{mainlem}

\begin{proof}
Let $\epsilon_\sigma$, $\sigma \in \omlom$, be such that $3 \sum \{ \epsilon_\sigma : \sigma \in \omlom \} \leq \epsilon$
and $\sum \{ \epsilon_\tau : \sigma \subsetneq \tau \} \leq \epsilon_\sigma$ for all $\sigma \in \omlom$.

We introduce a rank $\rk$ on $T$ as follows. 
\begin{itemize}
\item $\rk (\sigma) = 0$ if there is $S \leq T$ with $\stem (S) = \sigma$ such that $S$ decides $\max (\dot I)$. 
\item for $\alpha > 0$: $\rk (\sigma) = \alpha$ if $\neg \rk (\sigma) < \alpha$ and $\{ n \in \succ_T (\sigma) :
   \rk(\sigma \ha n) < \alpha \}$ is infinite.
\end{itemize}
A standard argument shows that every node has a rank and, by pruning $T$ appropriately, we may assume
that if $\sigma \subset \tau$ then
\begin{itemize}
\item $\rk (\sigma) > 0$ implies $\rk (\tau ) < \rk (\sigma)$,
\item $\rk (\sigma) = 0$ implies $\rk (\tau) = 0$.
\end{itemize}
Let $F$ be the set of all $\sigma$ with $\rk (\sigma) = 0$ and $\rk (\tau) > 0$ for all $\tau \subset \sigma$.
Then $F$ clearly is a front. Also note that by assumption we have $\rk (\stem (T)) > 0$.

Now fix $\sigma \in F$. By pure decision and by pruning $T_\sigma$, if necessary, we may assume that
there are $I^\sigma$, $B^\sigma$, $C^\sigma$, $(a^\sigma_k : k \in I^\sigma)$, $b^\sigma$, and $c^\sigma$
such that 
\begin{itemize}
\item $I^\sigma$ is finite, $B^\sigma \sub I^\sigma$, $C^\sigma = I^\sigma \sem B^\sigma$,
\item all $a^\sigma_k$ are reals between $0$ and $1$, $\sum_{k \in B^\sigma} a^\sigma_k = b^\sigma$, 
   $\sum_{k \in  C^\sigma}  a^\sigma_k = c^\sigma$,
\item $T_\sigma$ forces $\dot I = I^\sigma$, $\dot B = B^\sigma$, $\dot C = C^\sigma$, and 
   $|\dot a_k - a^\sigma_k | < {\epsilon_\sigma \over |I^\sigma| }$.
\end{itemize}
In particular, $T_\sigma$ forces $|\dot b - b^\sigma | , |\dot c - c^\sigma | < \epsilon_\sigma$. It follows that
$|b^\sigma - r | < \epsilon + \epsilon_\sigma$ and $|c^\sigma - s | < \epsilon + \epsilon_\sigma$.

By induction on rank and by pruning $T$ appropriately along the way,
we produce $\bar I_\sigma$, $\bar B_\sigma$, $\bar C_\sigma$, $\bar a_\sigma$, $r_\sigma$, and $s_\sigma$
(in case $\rk (\sigma) > 0$), as well as $I^\sigma$, $B^\sigma$, $C^\sigma$, $(a^\sigma_k : k \in I^\sigma)$, $b^\sigma$, and $c^\sigma$
such that
\begin{enumerate}
\item $I^\sigma$ is finite, $B^\sigma \sub I^\sigma$, $C^\sigma = I^\sigma \sem B^\sigma$,
\item all $a^\sigma_k$ are reals between $0$ and $1$, $\sum_{k \in B^\sigma} a^\sigma_k = b^\sigma$, 
   $\sum_{k \in  C^\sigma}  a^\sigma_k = c^\sigma$,
\item $b^\sigma < r + 2 \epsilon$, $c^\sigma < s + 2 \epsilon$,
\item $(\bar I_\sigma, \bar B_\sigma, \bar a_\sigma)$ is an $(r_\sigma,s_\sigma)$-sequence,
\item $\lim \{b^{\sigma \ha n} : n \in \succ_T (\sigma) \} \geq r_\sigma + b^\sigma > \lim \{  b^{\sigma \ha n} : n \in \succ_T (\sigma) \} - \epsilon_\sigma$,
   and similarly with $b,r$ replaced by $c,s$,
\item $|r_\sigma + b^\sigma - b^{\sigma \ha n} | < \epsilon_\sigma$ for all $n \in \succ_T (\sigma)$,
   and similarly with $b,r$ replaced by $c,s$,
\item $I_{\sigma,n}$ and $I^\sigma$ are disjoint and $ I^\sigma \cup I_{\sigma,n} \sub I^{\sigma \ha n}$,
\item $a_{\sigma, k} = a_k^{\sigma \ha n}$ for $k \in I_{\sigma,n}$, $a^\sigma_k = \lim \{ a^{\sigma \ha n}_k : n \in
   \succ_T (\sigma) \}$ for $k \in I^\sigma$,
\item $| a^\sigma_k - a^{\sigma \ha n}_k | < { \epsilon_\sigma \over | I^\sigma | }$ for $n \in \succ_T (\sigma)$ and $k \in I^\sigma$.
\end{enumerate}

In case $\rk (\sigma) = 0$, the necessary items have been produced above. So assume $\rk (\sigma) > 0$.
Then $\rk (\sigma \ha n) < \rk (\sigma)$ for all $n \in \succ_T (\sigma)$. In particular, we have $I^{\sigma \ha n}$,
$B^{\sigma \ha n}$, $C^{\sigma \ha n}$, $(a^{\sigma\ha n}_k : k \in I^{\sigma\ha n})$, $b^{\sigma\ha n}$, and $c^{\sigma \ha n}$
for $n \in \succ_T (\sigma)$. Using clause 3 and the fact that bounded sequences have convergent subsequences and pruning $\succ_T (\sigma)$,
if necessary, we may assume that $r^\sigma = \lim \{ b^{\sigma\ha n} : n \in \succ_T (\sigma) \}$ and 
$s^\sigma = \lim \{ c^{\sigma \ha n} : n \in \succ_T (\sigma) \}$ 
both exist. We may also assume that there are (possibly infinite and possibly empty) sets $\tilde I^\sigma$, $\tilde B^\sigma \sub \tilde I^\sigma$,
and $\tilde C^\sigma = \tilde I^\sigma \sem \tilde B^\sigma$ such that 
\begin{itemize}
\item if $k \in \tilde B^\sigma$, then $k \in B^{\sigma\ha n}$ for almost all $n \in \succ_T (\sigma)$,
\item if $k \in \tilde C^\sigma$, then $k \in C^{\sigma\ha n}$ for almost all $n \in \succ_T (\sigma)$, and
\item if $k \not\in \tilde I^\sigma$, then $k \not\in I^{\sigma\ha n}$ for almost all $n \in \succ_T (\sigma)$.
\end{itemize}
Next we may assume that for $k \in \tilde I^\sigma$, $a^\sigma_k = \lim \{ a^{\sigma\ha n}_k : n \in \succ_T (\sigma)$ and
$k \in I^{\sigma \ha n} \}$ exists. Let $\tilde b^\sigma = \sum \{ a^\sigma_k : k \in \tilde B^\sigma \}$ and
$\tilde c^\sigma = \sum \{  a^\sigma_k : k \in \tilde C^\sigma \}$. 

\begin{sclaim}
$r^\sigma \geq \tilde b^\sigma$ and $s^\sigma \geq \tilde c^\sigma$.
\end{sclaim}

\begin{proof}
Suppose this is false, and let $\delta > 0$ and $k_0$ be such that $\sum \{ a^\sigma_k : k < k_0$ and
$k \in \tilde B^\sigma \} > r^\sigma + \delta$. Let $n \in \succ_T (\sigma)$ be such that $r^\sigma > b^{\sigma\ha n} - {\delta \over 2}$
and $k \in B^{\sigma\ha n}$ and $| a^{\sigma \ha n}_k - a^\sigma_k | < {\delta \over 2 k_0}$ for all
$k \in \tilde B^\sigma$ with $k < k_0$. Then 
\begin{align*}
b^{\sigma \ha n} & \geq \sum \{ a^{\sigma\ha n}_k : k < k_0 \mbox{ and }
k \in \tilde B^\sigma \} \\ & \geq \sum \{ a^\sigma_k : k < k_0 \mbox{ and }k \in \tilde B^\sigma \} - {\delta \over 2} \\
& > r^\sigma + {\delta \over 2} > b^{\sigma \ha n},
\end{align*}
a contradiction. The proof of $s^\sigma \geq \tilde c^\sigma$ is analogous.
\end{proof}

Now let $r_\sigma = r^\sigma - \tilde b^\sigma$ and $s_\sigma = s^\sigma - \tilde c^\sigma$. 
Also let $I^\sigma$ be a finite initial segment of $\tilde I^\sigma$ such that if $B^\sigma = I^\sigma \cap \tilde B^\sigma$
and $C^\sigma = I^\sigma \cap \tilde C^\sigma$, then $b^\sigma := \sum \{  a_k^\sigma : k \in B^\sigma\}
> \tilde b^\sigma - \epsilon_\sigma$ and $c^\sigma := \sum \{  a_k^\sigma : k \in C^\sigma\}
> \tilde c^\sigma - \epsilon_\sigma$. If $\tilde I^\sigma$ is finite we may simply let $I^\sigma = \tilde I^\sigma$.
By pruning $T$ if necessary, we may also assume that $I^\sigma \sub I^{\sigma \ha n}$ for all
$n \in \succ_T (\sigma)$.

\begin{sclaim}
For every $\delta > 0$ and every large enough $k_0$ there is $n_0$ such that if $n \geq n_0$ belongs to
$\succ_T (\sigma)$, then $|\sum \{ a_k^{\sigma \ha n} : k \in B^{\sigma \ha n} $ and $k \geq k_0 \} - r_\sigma| < \delta$
and $|\sum \{ a_k^{\sigma \ha n} : k \in C^{\sigma \ha n} $ and $k \geq k_0 \} - s_\sigma| < \delta$.
\end{sclaim}

\begin{proof} 
Let $k_0$ be such that $| \sum \{ a_k^\sigma : k \in \tilde B^\sigma \cap k_0 \} - \tilde b^\sigma | < {\delta \over 3}$.
Note that for large enough $n \in \succ_T (\sigma)$, we have $B^{\sigma\ha n} \cap k_0 = \tilde B^\sigma \cap k_0$
and $C^{\sigma \ha n} \cap k_0 = \tilde C^\sigma \cap k_0$. Fix $n_0$ such that for all $n \geq n_0$ in $\succ_T (\sigma)$,
we have $|b^{\sigma \ha n} - r^\sigma | < {\delta \over 3 }$ and for all $k < k_0$ in $\tilde B^\sigma$, 
$|a^{\sigma \ha n}_k - a^\sigma_k | < {\delta \over 3 k_0}$.
Then
\begin{align*} 
\left| \sum \{  a_k^{\sigma \ha n} : \right. &  \left. k \in B^{\sigma \ha n}  \mbox{ and }k \geq k_0 \} - r_\sigma \right| \\ &= \left|\left(b^{\sigma \ha n} - 
\sum \{ a_k^{\sigma \ha n} : k \in B^{\sigma \ha n} \cap k_0 \} \right)- (r^\sigma - \tilde b^\sigma )\right|\\ & < {\delta \over 3} + 
\left|\sum \{ a_k^{\sigma \ha n} : k \in B^{\sigma \ha n} \cap k_0 \} - \sum \{ a_k^{\sigma} : k \in B^{\sigma \ha n} \cap k_0 \}\right|
\\ & +  \left|\sum \{ a_k^{\sigma} : k \in B^{\sigma \ha n} \cap k_0 \} - \tilde b^\sigma\right| <  3 {\delta \over 3} = \delta 
\end{align*}
as required. Similarly for $C^{\sigma \ha n}$ and $s_\sigma$.
\end{proof}

Thus, by pruning $\succ_T (\sigma)$ if necessary, we may find $I_{\sigma, n} \sub I^{\sigma \ha n} \sem I^\sigma$,
$B_{\sigma,n} = I_{\sigma,n} \cap B^{\sigma\ha n}$, and $C_{\sigma,n} = I_{\sigma,n} \cap C^{\sigma\ha n}$ such that
$\max (I_{\sigma,n} ) < \min (I_{\sigma,m})$ for $n < m$ in $\succ_T (\sigma)$ and, letting $b_{\sigma,n} = \sum
\{ a_k^{\sigma \ha n} : k \in B_{\sigma,n} \}$ and $c_{\sigma,n} = \sum\{ a_k^{\sigma \ha n} : k \in C_{\sigma,n} \}$
for $n \in \succ_T (\sigma)$, we have $\lim \{ b_{\sigma,n} : n \in \succ_T (\sigma) \} = r_\sigma$ and 
$\lim \{ c_{\sigma,n} : n \in \succ_T (\sigma) \} = s_\sigma$. In particular, clause 4 is satisfied. Let $a_{\sigma,k} = a_k^{\sigma \ha n}$
for $k \in I_{\sigma,n}$. Note that \[r^\sigma = \tilde b^\sigma + r_\sigma \geq b^\sigma + r_\sigma > \tilde b^\sigma - \epsilon_\sigma
+ r_\sigma = r^\sigma - \epsilon_\sigma,\] and similarly with $b$ and $r$ replaced by $c$ and $s$, 
so that clause 5 holds. Pruning $\succ_T (\sigma)$ if necessary,
clause 6 follows, and clause 9 can be guaranteed for the same reason. 
Clauses 1, 2, 3, 7, and 8 are obvious by definition. This completes the recursive construction.

Note that if $\sigma = \stem (T)$ (the final step of the recursion), then, by the assumption that no $S \leq_0 T$ decides $\min (\dot I)$,
we necessarily must have $I^\sigma = B^\sigma = C^\sigma = \emptyset$ and $b^\sigma = c^\sigma = 0$.

Now let $S$ be the tree obtained from $T$ by the various pruning operations described above.

Applying repeatedly clause 7, we see that for $\sigma \in F$ we have 
\[  \bigcup \{ I_{\sigma \re j , \sigma (j) } : | \stem (T)| \leq j < |\sigma | \} \sub I^\sigma, \]
with an analogous inclusion relation holding for $B^\sigma$ and $C^\sigma$. Similar considerations give us:

\begin{sclaim}   \label{claim3}
For $\sigma \in F$, \[\left| \sum \{ r_{\sigma \re j} : | \stem (T) | \leq j < |\sigma| \} - b^\sigma \right| < \sum \{ \epsilon_{\sigma \re j} :
| \stem (T) | \leq j < |\sigma| \}\] and similarly with $r$ and $b$ replaced by $s$ and $c$.
\end{sclaim}

\begin{proof}
Using that $b^{\stem (T)} = 0$, we see
\begin{align*}
\left| \sum r_{\sigma \re j} - b^\sigma \right| &= \left| \sum \left( r_{\sigma \re j} + b^{\sigma \re j} \right) - \sum b^{\sigma \re j} - b^\sigma \right| \\
&= \left| \sum \left( r_{\sigma \re j} + b^{\sigma \re j} \right) - \sum b^{\sigma \re (j+1)} \right| \\ & \leq \sum \left| r_{\sigma \re j} + b^{\sigma \re j}
- b^{\sigma \re (j+1)} \right|  < \sum \epsilon_{\sigma \re j} 
\end{align*}
where all sums are taken over $j$ with $| \stem (T) | \leq j < |\sigma|$ and where the last inequality holds by repeatedly applying clause 6.
\end{proof}

We also obtain:

\begin{sclaim}   \label{claim4}
For $\sigma \in F$, $j$ with $|\stem (T)| \leq j < |\sigma|$ and $k \in I_{\sigma \re j, \sigma (j) }$, 
\[ \left| a^\sigma_k - a_{\sigma \re j ,k} \right|  < { \epsilon_{\sigma \re j }  \over | I^{\sigma \re (j+1)} | }. \] 
\end{sclaim}

\begin{proof}
Since $a_{\sigma \re j ,k} = a_k^{\sigma \re ( j + 1)}$, we see
\begin{align*}
\left| a^\sigma_k - a_{\sigma \re j ,k} \right| &  \leq \sum \left\{ \left| a^{\sigma \re (i+1)}_k - a^{\sigma \re i}_k \right|  :  j < i < |\sigma | \right\} \\
& < \sum \left\{ { \epsilon_{\sigma \re i} \over | I^{\sigma \re i} | }: j < i < |\sigma | \right\} \\
& \leq 
 { \sum \left\{  \epsilon_{\sigma \re i}  : j < i < |\sigma | \right\}   \over | I^{\sigma \re (j+1)}| } < { \epsilon_{\sigma \re j }\over |I^{\sigma \re (j+1)} |}
\end{align*}
by clause 9, because $I^{\sigma \re (j+1)} \sub I^{\sigma \re i}$ for $j < i < |\sigma|$, and because $\sum \{  \epsilon_{\sigma \re i}  : 
j < i < |\sigma | \} < \epsilon_{\sigma \re j}$.
\end{proof}

Now assume $D \in \omoms$ almost splits all $(\bar I_\sigma, \bar B_\sigma, \bar a_\sigma)$ with $\stem (T) \sub \sigma \subsetneq \tau$
for some $\tau \in F$. This means in particular that for each such $\sigma$ there are infinitely many
$n \in \succ_T (\sigma)$ such that 
\[ \left| \sum \{ a_{\sigma , k}  : k \in B_{\sigma , n} \cap D \} - r_\sigma \right| < \epsilon_\sigma \mbox{ and } 
   \sum \{ a_{\sigma , k}  : k \in C_{\sigma , n} \cap D \} < \epsilon_\sigma. \]
Hence we can easily build $S' \leq_0 S$ such that $S'_\sigma = S_\sigma$ and
\[ \left| \sum \{ a_{\sigma \re j, k}  : k \in B_{\sigma \re j, \sigma (j)} \cap D \} - r_{\sigma \re j} \right| < \epsilon_{\sigma \re j} \qquad \mbox{ and } \]\[
   \sum \{ a_{\sigma \re j, k}  : k \in C_{\sigma \re j, \sigma (j)} \cap D \}  < \epsilon_{\sigma \re j} \]
hold for all $j$ with $|\stem (T) | \leq j < |\sigma|$ and all $\sigma \in F \cap S'$.
We need to prove that $S'$ is as required.

Fix $\sigma \in F \cap S'$. It clearly suffices to show that 
\[ S_\sigma \forces \left| \sum_{k \in \dot B \cap D} \dot a_k - \dot b \right| < \epsilon \mbox{ and }  \sum_{k \in \dot C \cap D} \dot a_k < \epsilon. \]
We only do the first part; the second part is similar and simpler.
Also note that $\sum_{k \in \dot B \cap D} \dot a_k$ is obviously forced to be less or equal than $\dot b$.
Hence the next claim completes the proof of the main lemma.

\begin{sclaim}
$S_\sigma$ forces $\sum_{k \in \dot B \cap D} \dot a_k > \dot b - \epsilon$.
\end{sclaim}

\begin{proof}
Let $N = |\stem (T)|$. Forgetting about names for the moment we compute
\begin{align*}
\sum_{k \in B^\sigma \cap D} a^\sigma_k & \geq \sum_{N \leq j < |\sigma|} \ \ \sum_{k \in B_{\sigma \re j, \sigma (j) }\cap D} a^\sigma_k \\
   & > \sum_{N \leq j < |\sigma|} \ \ \sum_{k \in B_{\sigma \re j, \sigma (j) } \cap D} a_{\sigma \re j, k} - \sum_{N \leq j < |\sigma|}
   \epsilon_{\sigma \re j} \\
   & > \sum_{N \leq j < |\sigma|}  r_{\sigma \re j} - 2 \sum_{N \leq j < |\sigma|}  \epsilon_{\sigma \re j} \ 
   > \ b^\sigma - 3 \sum_{N\leq j < |\sigma|}  \epsilon_{\sigma \re j} 
\end{align*}   
where the first inequality holds by $\bigcup \{ B_{\sigma \re j , \sigma (j) } : N \leq j < |\sigma | \} \sub B^\sigma$,
the second by Claim~\ref{claim4} and because $B_{\sigma \re j, \sigma (j)} \sub I_{\sigma \re j, \sigma (j) } \sub I^{\sigma \re (j+1)}$, 
the third by the choice of $S'$, and the forth by Claim~\ref{claim3}.
Hence
\begin{align*}
S_\sigma \forces \sum_{k \in \dot B \cap D} \dot a_k & > \sum_{k \in B^\sigma \cap D} a^\sigma_k - \epsilon_\sigma 
\ > \ b^\sigma - 3 \sum_{N \leq j < |\sigma|}  \epsilon_{\sigma \re j} - \epsilon_\sigma \\ & > \ \dot b 
- 3 \sum_{N \leq j < |\sigma|}  \epsilon_{\sigma \re j} - 2  \epsilon_\sigma  \ > \ \dot b - 
3 \sum_{N \leq j \leq |\sigma|}  \epsilon_{\sigma \re j} \ > \ \dot b - \epsilon
\end{align*}   
as required.
\end{proof}
\end{proof}

\begin{la}
Let $( \dot{\bar I}^m , \dot{\bar B}^m, \dot{\bar a}^m )$, $m \in \omega$, be $\LL$-names for $(\dot r^m, \dot s^m)$-sequences. Let
$T \in \LL$. Then there are $S \leq_0 T$ and $(r^j , s^j)$-sequences $(\bar I^j, \bar B^j, \bar a^j)$, $j \in\omega$,
such that whenever $D \in \omoms$ almost splits all $(\bar I^j , \bar B^j , \bar a^j)$, then there is $S' \leq_0 S$ such that
\[ S' \forces D \mbox{ almost splits all } (\dot{\bar I}^m , \dot{\bar B}^m, \dot{\bar a}^m). \]
\end{la}

\begin{proof}
By pruning the names for the sequences, if necessary, we may assume that for each $m \in \omega$,
the function $n \mapsto \min (\dot I^m_n)$ dominates the Laver generic. This implies in particular that
for each $\sigma \in T$ below the stem and each $T' \leq_0 T_\sigma$, $T'$ decides $\min (\dot I^m_n)$ for
only finitely many $n$. For the same reason, we may assume that for every $m$, 
\[ T \forces |\dot r^m - \dot b_n^m | < {1 \over 2^n} \mbox{ and }  |\dot s^m - \dot c_n^m | < {1 \over 2^n}. \]
We will build the tree $S$ as the fusion of a sequence $(T_\ell : \ell \in \omega)$ by recursively
specifying fronts $F_\ell$ for $\ell \in \omega$. The $F_\ell$ and $T_\ell$ will satisfy
\begin{itemize}
\item $T_0 = T$, $F_0 = \{ \stem (T) \}$, 
\item $T_{\ell + 1} \leq_0 T_\ell$ and $F_\ell$ is a front in $T_{\ell '}$ for all $\ell ' \geq \ell$,
\item for every $\sigma \in F_{\ell + 1}$ there is a unique $\tau \in F_\ell$ such that $\tau \subsetneq \sigma$.
\end{itemize}
In the end we shall put $S = \{ \sigma \in T: \exists \ell \; \exists \tau \in F_\ell \; (\sigma \sub \tau) \} = \bigcap_\ell T_\ell$, the tree
{\em generated by} the fronts $F_\ell$. Once $F_\ell$ has been defined (at stage $\ell$) it will not be changed 
anymore and all the later pruning of the original $T$ will occur below the nodes of $F_\ell$.

Let $e : \omega \times \omega \to \omega $ be a bijection. At stage $\ell$ of the construction, that is,
when $F_\ell$ and $T_\ell$ are given, we will basically apply the main lemma to the names $\dot I^m_n$, $\dot B^m_n$,
$\dot C^m_n$, $(\dot a_k^m : k \in \dot I^m_n)$, $\dot b^m_n = \sum_{k \in \dot B^m_n} \dot a_k^m$, and
$\dot  c^m_n = \sum_{k \in \dot C^m_n} \dot a_k^m$ where $\ell = e (m,n)$ and obtain $F_{\ell + 1}$. 
More explicitly, do the following. Fix $\sigma \in F_\ell$. We may assume no subtree of $(T_\ell)_\sigma$ 
with the same stem decides the value of $\min (\dot I^m_n)$; otherwise replace $\dot I^m_n$ by some
$\dot I^m_{n'}$ for an appropriate $n' > n$, see the discussion at the beginning of the preceding paragraph.
Furthermore we may assume that for some real numbers $r^m_\sigma$ and $s^m_\sigma$,
\[ (T_\ell)_\sigma \forces | \dot r^m -  r^m_\sigma | < {1 \over 2^n} \mbox{ and } | \dot s^m - s^m_\sigma | < {1 \over 2^n}. \]
The point is that by pruning $T$ below $\sigma$ appropriately, we can find a front $F$ below $\sigma$
such that all $\tau \in F$ have this property with $\sigma$ replaced by $\tau$. We can then remove $\sigma$ from $F_\ell$
and replace it by $F$. So assume without loss of generality that $\sigma$ already has this property.
As a consequence we obtain
\[ (T_\ell)_\sigma \forces | \dot b_n^m -  r^m_\sigma | < {1 \over 2^{n-1}} \mbox{ and } | \dot c_n^m - s^m_\sigma | < {1 \over 2^{n-1}}. \]

Thus we may apply the main lemma with $\epsilon  = {1 \over 2^{n-1}}$ and obtain a tree $T^\sigma \leq_0
(T_\ell)_\sigma$ with a front $F^\sigma$ and sequences $\bar I_\tau = (I_{\tau,n} : n \in \succ_{T^\sigma} (\tau))$,
$\bar B_\tau = (B_{\tau,n} : n \in \succ_{T^\sigma} (\tau))$, and $\bar a_\tau = (a_{\tau,k} : k \in \omega)$, and
reals $r_\tau$ and $s_\tau$ for all $\tau \in T^\sigma$ with $\sigma \sub \tau$ and $\tau \subsetneq \rho$ for some
$\rho \in F^\sigma$ satisfying the conclusion of the main lemma. Now unfix $\sigma \in F_\ell$ and let $T_{\ell + 1} 
= \bigcup \{ T^\sigma : \sigma \in F_\ell \}$ and $F_{\ell + 1} = \bigcup \{ F^\sigma : \sigma \in F_\ell \}$. 
Clearly $T_{\ell + 1}$ and $F_{\ell + 1}$ satisfy the requirements. This completes the recursive step of the construction
and, as mentioned, we put $S = \{ \sigma \in T: \exists \ell \; \exists \tau \in F_\ell \; (\sigma \sub \tau) \} = \bigcap_\ell T_\ell$.
Clearly $S \leq_0 T$. Also let $((\bar I^j, \bar B^j, \bar a^j) : j \in \omega )$ list all sequences 
$(\bar I_\tau , \bar B_\tau , \bar a_\tau)$, $\tau \in S$, $\stem (S) \sub \tau$, produced along the way. 

Now assume that $D \in \omoms$ almost splits all $(\bar I^j, \bar B^j, \bar a^j)$. By recursion on $\ell$ we produce
$S_\ell$, $\ell \in \omega$, such that $S'$ will be the fusion of the $S_\ell$ and the following hold:
\begin{itemize}
\item $S_0 = S$, 
\item $S_{\ell +1} \leq_0 S_\ell$ and $F_\ell \cap S_{\ell + 1} = F_\ell \cap S_\ell$,
\item for all $\sigma \in F_\ell$, $(S_\ell)_\sigma = S_\sigma$,
\item $S_{\ell +1} \forces | \sum_{k \in \dot B^m_n \cap D} \dot a^m_k - \dot r^m | < {1 \over 2^{n-3}}$ and 
   $\sum_{k \in \dot C^m_n \cap D} \dot a^m_k  < {1 \over 2^{n-3}}$ where $\ell = e (m,n)$.
\end{itemize}
Suppose $S_\ell$ has been produced and we construct $S_{\ell + 1}$. Fix $\sigma \in F_\ell \cap S_\ell$. 
By the main lemma, we know that there is $S^\sigma \leq_0 (S_\ell)_\sigma$ such that $S^\sigma_\tau
= (S_\ell)_\tau = S_\tau$ for all $\tau \in F_{\ell + 1} \cap S^\sigma$ and 
\[ S^\sigma \forces \left| \sum_{k \in \dot B^m_n \cap D} \dot a^m_k - r^m_\sigma \right| < {1 \over 2^{n-2}} \mbox{ and }
   \sum_{k \in \dot C^m_n \cap D} \dot a^m_k  < {1 \over 2^{n-2}}. \]
Unfix $\sigma \in F_\ell \cap S_\ell$ and let $S_{\ell + 1}  = \bigcup \{ S^\sigma : \sigma \in F_\ell \cap S_\ell\}$.
It is now easy to see that $S_{\ell + 1}$ forces the required statements. 

Let $S' = \bigcap_\ell S_\ell$. We now see
\[ S' \forces \lim_n \left( \sum_{k \in \dot B^m_n \cap D} \dot a^m_k \right) = \dot r^m \mbox{ and } 
   \lim_n \left( \sum_{k \in \dot C^m_n \cap D} \dot a^m_k \right) = 0 \]
for all $m \in \omega$. This completes the proof of the lemma.
\end{proof}

By this lemma we know that adding one Laver real preserves any almost splitting family of the ground model.

We now need to deal with the iteration. This is clearly a case for $G_\delta$ preservation. It would be natural
to apply the second preservation theorem in the Bartoszy\'nski-Judah book~\cite[Theorem 6.1.18]{bj}, but it
is not clear whether condition 2 of~\cite[Definition 6.1.17]{bj} is satisfied. We therefore use the more natural
version of $G_\delta$ preservation due to Eisworth~\cite{Ei94}. 

Let $\PP$ be a proper forcing, let $\lambda$ be a sufficiently large cardinal, and assume $N \prec H(\lambda)$ is countable 
with $\PP \in N$. Let $\B$ be a Borel set with $N \cap \omom \sub \B$. Say $\PP$ {\em preserves} $(N,\B)$ if for each
$p \in N \cap \PP$ there is an $(N,\PP)$-generic condition $q \leq p$ such that $q \forces N[\dot G] \cap \omom
\sub \B$~\cite[Definition 3.2.1]{Ei94}. 

Let $\P_0$ be the collection of all quintuples $(\bar I,\bar B, \bar a, r,s)$. Identifying real numbers with their
binary expansions, we may construe $\P_0$ as a closed subset of $\omom$ and thus as a Polish space itself. In fact,
under this identification, $\P_0$ is homeomorphic to $\omom$. Next, let $\P \sub \P_0$ be the collection of
all $(\bar I, \bar B, \bar a, r,s)$ such that $|b_n - r| < {1 \over n}$ and $|c_n - s | < {1 \over n}$ for all $n$.
If $(\bar I, \bar B, \bar a, r,s) \in \P$, then $(\bar I, \bar B , \bar a )$ is an $(r,s)$-sequence. Conversely,
if $(\bar I, \bar B , \bar a )$ is an $(r,s)$-sequence, then for some subsequence $(\bar I ', \bar B ' , \bar a )$,
$(\bar I ', \bar B ', \bar a, r,s)$ belongs to $\P$. The reason for considering $\P$ is that it is closed in $\P_0$ and thus
again Polish and also homeomorphic to $\omom$. On the other hand, as remarked earlier, for the phenomenon
of almost splitting, it suffices to consider appropriate subsequences.

Let $D \in \omoms$. Let $\B_D$ be the collection of all $(\bar I, \bar B, \bar a, r,s) \in \P$ such that $D$ splits
the $(r,s)$-sequence $(\bar I, \bar B, \bar a)$. Letting $\B_D^n$ be the set of $(\bar I, \bar B, \bar a, r,s) \in \P$ such that
for some $n' \geq n$ we have $|\sum_{k \in B_{n'} \cap D} a_k - r | < {1 \over n}$ and $\sum_{k \in C_{n'} \cap D} 
a_k < {1 \over n}$, we see that each $\B_D^n$ is open and $\B_D = \bigcap_n \B_D^n$. A fortiori, $\B_D$ is a 
$G_\delta$ set. 

\begin{la}   \label{models-Laver}
Let $N \prec H(\lambda)$ be countable with $\LL \in N$. Also assume $D \in \omom$ is such that $N \cap \P \sub \B_D$.
Then $\LL$ preserves $(N,\B_D)$.
\end{la}

\begin{proof}
This is what the previous lemma gives us in this new context. More explicitly let $T \in N \cap \LL$. Also let
$(\dot{\bar I}^m , \dot{\bar B}^m, \dot{\bar a}^m, \dot r^m, \dot s^m)$ enumerate the $\LL$-names of members of $\P$ 
belonging to $N$. In a fusion argument we constructed $S \leq_0 T$ and $(\bar I^j, \bar B^j , \bar a^j, r^j , s^j) \in \P$
in the previous proof. While the whole construction takes place outside $N$, any finite initial segment is in $N$.
In particular all the $(\bar I^j, \bar B^j , \bar a^j, r^j , s^j)$ belong to $N$. Furthermore, by interleaving this construction
with the usual construction guaranteeing genericity, we may assume that $S$ is actually $(N,\LL)$-generic.
Since $D$ almost splits $N \cap \P$ it almost splits in particular all $(\bar I^j, \bar B^j , \bar a^j)$. 
Therefore, if $S' \leq_0 S$ is as in the conclusion of the lemma, it forces that $D$ almost splits all 
$(\dot{\bar I}^m , \dot{\bar B}^m, \dot{\bar a}^m)$. That is, $S' \forces N[\dot G] \cap \P \sub \B_D$,
as required.
\end{proof}

We now apply~\cite[Corollary 3.2.4]{Ei94}.

\begin{thm}[Eisworth]   \label{models-limit}
Suppose $\PP = ( \PP_\alpha, \dot \QQ_\alpha : \alpha < \kappa )$ is a countable support iteration of proper 
forcings. Let $\lambda$ be sufficiently large and let $N \prec H(\lambda)$ be countable with $\PP \in N$.
Let $\B$ be a $G_\delta$ set such that $N \cap \omom \sub \B$, and assume that for each $\alpha < \kappa$,
$\forces_\alpha`` \dot \QQ_\alpha$ preserves $(N [\dot G_\alpha], \B )"$. Then $\PP_\kappa$ preserves
$(N, \B)$.
\end{thm}

\begin{thm}
$\salmost = \aleph_1$ in the Laver model. In particular, $\salmost < \ger b$ is consistent.
\end{thm}

\begin{proof}
Let $(\dot{\bar I}, \dot{\bar B}, \dot{\bar a}, \dot r, \dot s)$ be an $\LL_{\omega_2}$-name for a 
member of $\P$. Also let $p \in \LL_{\omega_2}$.
Let $N \prec H(\lambda)$ be countable with $\PP, p, (\dot{\bar I}, \dot{\bar B}, \dot{\bar a}, \dot r, \dot s) \in N$.
Let $D \in \omoms$ be such that $N \cap \P \sub \B_D$. By induction, using Lemma~\ref{models-Laver}
for the iterands and Theorem~\ref{models-limit} for the iteration we see that all $\PP_\alpha$
preserve $(N,\B_D)$. Hence we may find an $(N, \LL_{\omega_2})$-generic condition $q \leq p$ such that
$q \forces N[\dot G_{\omega_2}] \cap \P \sub \B_D$. In particular $q$ forces that $D$ almost splits
$(\dot{\bar I}, \dot{\bar B}, \dot{\bar a})$, as required.
\end{proof}

\section{The subrearrangement number: a characterization of $\min \{ \ss,\rr \}$}

In this section we will consider another cardinal invariant related to the subseries numbers and the rearrangement numbers, which we call the \emph{subrearrangement} number. One may think of this number as combining the idea behind the subseries numbers and the rearrangement numbers: first one chooses a subseries, and then one permutes the terms of this subseries, in order to test whether a series converges conditionally. This definition was suggested by Rahman Mohammadpour on MathOverflow \cite{Rahman}.

\begin{df}      \label{df-sr}
 $\sr$ is the smallest cardinality of any family $\mathcal F$ of
      injective functions $\N \to \N$ such that, for every conditionally
      convergent series $\sum_{n \in \N}a_n$ of real numbers, there is some
      $f \in \mathcal F$ such that the series $\sum_{n \in \N}a_{f(n)}$ diverges.
\end{df}

If $f: \N \to \N$ is an injective function, then the series $\sum_{n \in \N}a_{f(n)}$ is a rearrangement of the subseries $\sum_{n \in f[\N]}a_n$. Conversely, if $A$ is an infinite subset of $\N$ and $p: A \rightarrow A$ is a permutation of $A$, then there is an injective function $f: \N \to \N$ such that
$$\sum_{n \in \N}a_{f(n)} = \sum_{n \in A}a_{p(n)},$$
namely $f(n) = p(\text{the }n^{\mathrm{th}}\text{ element of }A)$. This is why injective functions are used in the definition of $\sr$: they are merely a convenient way of modeling the idea of first taking a subseries and then rearranging it. Alternatively, one may consider injective functions as modeling the process of first taking a rearrangement, and then taking a subseries of that rearrangement. The order in which one does these things is irrelevant.

Of course, one could also consider different types of subrearrangement numbers, defining $\sri$ and $\sro$ in analogy with $\ssi$ and $\sso$ according to the manner in which the series defined by our injective functions diverges. We will not deal with these variants here, but will confine our attention to $\sr$ only.

The main theorem of this section, the last theorem of this paper, states that $\sr = \min \{\ss,\rr\}$. Thus $\sr$ does not really constitute a new cardinal characteristic, but merely an interesting alternative description for $\min \{\ss,\rr\}$. Let us begin with the easy direction:

\begin{thm}\label{thm:sr1}
$\sr \leq \rr$ and $\sr \leq \ss$.
\end{thm}
\begin{pf}
Suppose $\mathcal C$ satisfies the definition of $\rr$; i.e., $\mathcal C$ is a family of permutations of $\N$ witnessing that every conditionally convergent series has a divergent rearrangement. Because every permutation is injective, $\mathcal C$ also satisfies the definition of $\sr$. This shows $\sr \leq \rr$.

Suppose $\mathcal A$ is a family of subsets of $\N$ witnessing that every conditionally convergent series has a divergent subseries. Because finite sets are useless for this purpose, we may assume without loss of generality that every $A \in \mathcal A$ is infinite. Let $\mathcal F = \{e_A : A \in \mathcal A\}$, where $e_A$ denotes the unique increasing enumeration of $A$. Then $\mathcal F$ will satisfy the definition of the subrearrangement number, and this shows $\sr \leq \ss$.
\end{pf}

The proof that $\min\{\rr,\ss\} \leq \sr$ breaks into two cases, according to whether or not $\sr < \ger b$. Note that $\sr < \ger b$ is consistent: combining the previous theorem with the results from Section~\ref{sec:Laver}, we see that this inequality holds in the Laver model. Because the proofs of these two cases do not overlap, we will break them up into two separate theorems below.

Before tackling the first of these two cases, we will need a lemma providing an alternative characterization of $\ger b$. Let us say that a set $A \subseteq \N$ is \emph{preserved} by an injective function $f: \N \to \N$ if $f$ does not change the relative order of members of $A$ except for finitely many elements; that is, for all but finitely many $x,y \in A$, we have $x < y$ if and only if $f(x) < f(y)$. If $A$ is not preserved by $f$, we say that $A$ is \emph{jumbled} by $f$.

\begin{la}\label{lem:jumbling}
The unbounding number $\ger b$ is the smallest cardinality of a family $\mathcal F$ of injective functions $\N \to \N$ with the property that every infinite $A \subseteq \N$ is jumbled by some $f \in \mathcal F$.
\end{la}
\begin{pf}
A variant of this lemma was proved as Theorem 16 in \cite{rearrangement}. The variant there dealt only with bijections $\N \to \N$ rather than injections. However, the proof given there does not use the surjectivity of these functions at any point, so substituting the word ``injection'' for every instance of the word ``bijection'' in that proof provides a proof of the present lemma.
\end{pf}

\begin{thm}\label{thm:sr2}
If $\sr < \ger b$, then $\sr = \ss$.
\end{thm}
\begin{pf}
Suppose $\sr < \ger b$. Let $\mathcal F$ be a family of injective functions ${\N \to \N}$ with the property that, for every conditionally convergent series $\sum_{n \in \N}a_n$, the series $\sum_{n \in \N}a_{f(n)}$ diverges for some $f \in \mathcal F$. Moreover, let us suppose $|\mathcal F| < \ger b$. We will find a family $\mathcal A$ of subsets of $\N$ such that $|\mathcal A| \leq |\mathcal F|$ and, for every conditionally convergent series $\sum_{n \in \N}a_n$, the subseries $\sum_{n \in A}a_n$ diverges for some $A \in \mathcal A$. This will prove $\ss \leq \sr$, and this suffices to prove the theorem because the reverse inequality is already proved.

By Lemma~\ref{lem:jumbling}, there is an infinite $B \subseteq \N$ such that, for all $f \in \mathcal F$, we have $x < y$ if and only if $f(x) < f(y)$ for all but finitely many members of $B$.

The idea of the proof is as follows. Given a conditionally convergent series $\sum_{n \in \N}a_n$, we may insert a large number of zeros, as in the proof of Theorem~\ref{aboves:thm}, to obtain a new series that is identical to the original except that its nonzero terms occur only on $B$. Some $f \in \mathcal F$ must make this new series diverge. However, by our choice of $B$, the function $f$ does not significantly rearrange the terms of the new series. If $f$ cannot make the new series diverge by rearranging its terms, then it must make the new series diverge by picking out a divergent subseries. Thus, by writing the terms of the series on $B$, we can use $f \in \mathcal F$ to find a divergent subseries.

Let $e_B$ denote the unique increasing enumeration of $B$. For every $f \in \mathcal F$, define
$$A_f = e_B^{-1}[B \cap f[\N]] = \{ n \in \N : e_B(n) \in f[\N] \}$$
and let
$$\mathcal A = \{ A_f : f \in \mathcal F \}.$$
Clearly $|\mathcal A| \leq |\mathcal F|$, and we claim that this family $\mathcal A$ is as required.

Let $\sum_{n \in \N}a_n$ be a conditionally convergent series. As in the proof of Theorem~\ref{aboves:thm}, define a new series $\sum_{n \in \N}c_n$ by setting
\begin{align*}
c_n =
\begin{cases}
a_k \ \  & \text{if } n = e_B(k) \\
0 \ \  & \text{if } n \notin B.
\end{cases}
\end{align*}

The series $\sum_{n \in \N}c_n$ has the same nonzero terms as $\sum_{n \in \N}a_n$, in the same order; the only difference is that many zeros have been inserted. In particular, $\sum_{n \in \N}c_n$ is conditionally convergent, so there is some $f \in \mathcal F$ such that the series $\sum_{n \in \N}c_{f(n)}$ diverges.

Consider the series $\sum_{n \in A_f}a_n$. By the definition of $A_f$ and of the $c_n$, we have
$$\sum_{n \in A_f}a_n = \sum_{n \in B \cap f[\N]}a_{e_B^{-1}(n)} = \sum_{n \in B \cap f[\N]}c_n.$$
Furthermore, because $c_n = 0$ whenever $n \notin B$ and because $f$ is order-preserving on $B$ with only finitely many exceptions, we have
$$\sum_{n \in B \cap f[\N]}c_n = \sum_{n \in f[\N]}c_n = \sum_{n \in \N}c_{f(n)}$$
so that $\sum_{n \in A_f}a_n$ diverges, as required.
\end{pf}

\begin{thm}\label{thm:sr3}
If $\sr \geq \ger b$, then $\sr = \rr$.
\end{thm}
\begin{pf}
Suppose $\sr \geq \ger b$. Let $\mathcal F$ be a family of injective functions ${\N \to \N}$ with the property that, for every conditionally convergent series $\sum_{n \in \N}a_n$, the series $\sum_{n \in \N}a_{f(n)}$ diverges for some $f \in \mathcal F$. 

We already know from Theorem~\ref{thm:sr1} that $\sr \leq \rr$, so it remains to prove the reverse inequality. The proof that $\rr \leq \sr$ is similar to that of Theorem~\ref{ssvsrr:thm}, part $(1)$, where we proved that $\rr \leq \max \{ \ss, \ger b \}$, but with some details different in this case.

Let $\mathcal B$ be a family of subsets of $\N$ with the property that no single subset of $\N$ is sparser than every $B \in \mathcal B$, and such that $|\mathcal B| \leq |\mathcal F|$. This is possible by Lemma~\ref{lem:b&d} and the assumption that $\sr \geq \ger b$. For reasons that become apparent later in the proof, let us assume that $\mathcal B$ is closed under the operation $\{b_1,b_2,b_3,\dots\} \mapsto \{b_1+1,b_2+2,b_3+3,\dots\}$, where $b_1 < b_2 < b_3 < \dots$.

We will find a family $\mathcal C$ of permutations of $\N$ such that $|\mathcal C| \leq {|\mathcal F| \cdot |\mathcal B|} \break = |\mathcal F|$ and, for every conditionally convergent series $\sum_{n \in \N}a_n$, the rearrangement $\sum_{n \in \N}a_{p(n)}$ diverges. This suffices to prove the theorem.

We will define the members of $\mathcal C$ from the members of $\mathcal F$ and $\mathcal B$, as one might expect, but the definition requires two cases. Specifically, let us partition $\mathcal F$ into two sets:
$$\mathcal F_0 = \{f \in \mathcal F : \N \setminus f[\N] \text{ is finite}\}$$
$$\mathcal F_1 = \{f \in \mathcal F : \N \setminus f[\N] \text{ is infinite}\}$$
and deal with each of these sets separately.

For each $f \in \mathcal F_0$, let $k_f = |\N \setminus f[\N]|$ and define $p_f: \N \to \N$ by
\begin{align*}
p_f(n) = 
\begin{cases}
\text{the }n^{\mathrm{th}}\text{ member of }\N \setminus f[\N] & \text{ if }n \leq k_f \\
f(n - k_f) & \text{ if }n > k_f
\end{cases}
\end{align*}
In other words, $p_f$ is the permutation of $\N$ that does exactly what $f$ does, but it takes the $k_f$ members of $\N \setminus f[\N]$ and sticks them at the beginning.

For $f \in \mathcal F_1$, observe that both $f[\N]$ and $\N \setminus f[\N]$ are infinite. For each $f \in \mathcal F_1$ and each $B \in \mathcal B$, define $p_{f,B}: \N \to \N$ by
\begin{align*}
p_{f,B}(n) = 
\begin{cases}
f(k) \\ \qquad \qquad \qquad \qquad \text{ if } n \text{ is the } k^{\mathrm{th}} \text{ member of }\N \setminus B \\
\text{the } k^{\mathrm{th}} \text{ member of } \N \setminus f[\N] \\ \qquad \qquad \qquad \qquad \text{ if } n \text{ is the } k^{\mathrm{th}} \text{ member of } B.
\end{cases}
\end{align*}
In other words, $p_{f,B}$ is the permutation that writes the image of $f$ out onto $\N \setminus B$, without changing the order of things as determined by $f$, and then stretches out the complement of $f[\N]$ onto the sparse set $B$.

Define
$$\mathcal C = \{ p_f : f \in \mathcal F_0 \} \cup \{ p_{f,B} : f \in \mathcal F_1 \text{ and } B \in \mathcal B \}.$$
It is clear that $|\mathcal C| \leq |\mathcal F|+|\mathcal F| \cdot |\mathcal B| = |\mathcal F|$, so it remains to show that, for every conditionally convergent series $\sum_{n \in \N}a_n$, the rearrangement $\sum_{n \in \N}a_{p(n)}$ diverges for some $p \in \mathcal C$.

Let $\sum_{n \in \N}a_n$ be a conditionally convergent series, and fix $f \in \mathcal F$ such that $\sum_{n \in \N}a_{f(n)}$ diverges. If $f \in \mathcal F_0$, then $\sum_{n \in \N}a_{p_f(n)}$ diverges also, because this series is the same as $\sum_{n \in \N}a_{f(n)}$, except that it may include some finitely many extra terms at the beginning. As $p_f \in \mathcal C$, we have reached the desired conclusion in this case.

It remains to consider the case $f \in \mathcal F_1$. For this case, we will now, just as in the proof of Theorem~\ref{ssvsrr:thm},  define a sparse subset of $\N$ that is meant to capture the rate at which the series $\sum_{n \in \N}a_{f(n)}$ diverges. We consider three cases:
\begin{itemize}
\item If $\sum_{n \in \N}a_{f(n)} = \infty$, then the partial sums $\sum_{n \leq m}a_{f(n)}$ increase without bound. We may therefore find an increasing sequence $m_1,m_2,m_3,\dots$ of natural numbers such that
$$ \sum_{m_k < n \leq m_{k+1}}a_{f(n)} > 1$$
for all $k \in \N$.
\item If $\sum_{n \in \N}a_{f(n)} = -\infty$, then the partial sums $\sum_{n \leq m}a_{f(n)}$ decrease without bound. We may therefore find an increasing sequence $m_1,m_2,m_3,\dots$ of natural numbers such that
$$ \sum_{m_k < n \leq m_{k+1}}a_{f(n)} < -1$$
for all $k \in \N$.
\item If $\sum_{n \in \N}a_{f(n)}$ diverges by oscillation, then we may find some $c > 0$ such that the partial sums $\sum_{n \leq m}a_{f(n)}$ undergo infinitely many oscillations of size at least $c$. More precisely, we may find an increasing sequence 
$m_1, m_1', m_2, m_2', m_3, m_3', \dots$
of natural numbers such that, for every $k \in \N$,
$$ \sum_{m_k < n \leq m_k'}a_{f(n)} > c,$$
$$ \sum_{m_k' < n \leq m_{k+1}}a_{f(n)} < -c.$$
\end{itemize}

Let $M = \{ m_n : n \in \N \}$. By our choice of $\mathcal B$, we may find some $B_0 \in \mathcal B$ such that $M$ is not sparser than $B_0$. If $B_0 = \{b_1,b_2,b_3,\dots\}$, then let $B = \{b_1+1, b_2+2, b_3+3, \dots\}$, and recall that, by our assumptions on the family $\mathcal B$, $B \in \mathcal B$. We claim that the rearranged series $\sum_{n \in \N}a_{p_{f,B}(n)}$ diverges. As $p_{f,B} \in \mathcal C$, this will suffice to complete the proof.

By the definition of ``sparser than'' there are infinitely many values of $k$ such that the interval $(m_k,m_{k+1}]$ does not contain any members of $B_0$. By our choice of $p_{f,B}$, for each such interval, the terms
$$a_{f(m_k+1)}, a_{f(m_k+2)}, a_{f(m_k+3)}, \dots, a_{f(m_{k+1})}$$
will appear, in the order shown, in the rearranged series $\sum_{n \in \N}a_{p_{f,B}(n)}$. The proof of this is nearly identical to the proof of Claim 2 within the proof of Theorem~\ref{ssvsrr:thm}.

If $\sum_{n \in \N}a_{f(n)} = \infty$, then this observation, together with our choice of the $m_k$, guarantees that the partial sums of the rearranged series $\sum_{n \in \N}a_{p_{f,B}(n)}$ will infinitely often increase by $1$, so that $\sum_{n \in \N}a_{p_{f,B}(n)}$ diverges. Similarly, if $\sum_{n \in A}a_n = -\infty$ then the partial sums of the rearranged series $\sum_{n \in \N}a_{p_{f,B}(n)}$ will infinitely often decrease by $1$, again implying that $\sum_{n \in \N}a_{p_{f,B}(n)}$ diverges. Lastly, if $\sum_{n \in \N}a_{f(n)}$ diverges by oscillation, then there is some $c > 0$ such that the partial sums of the rearranged series $\sum_{n \in \N}a_{p_{f,B}(n)}$ will infinitely often oscillate by $c$, once again implying that $\sum_{n \in \N}a_{p_{f,B}(n)}$ diverges.
\end{pf}

\begin{thm}
$\sr = \min \{\ss,\rr\}$. More specifically,
\begin{align*}
\sr = 
\begin{cases}
\ss < \rr & \text{ if } \ss < \ger b \\
\rr \leq \ss & \text{ if } \ss \geq \ger b.
\end{cases}
\end{align*}
\end{thm}
\begin{pf}
This follows immediately from Theorems \ref{thm:sr1}, \ref{thm:sr2}, and \ref{thm:sr3}.
\end{pf}

As mentioned at the beginning of this section, the inequality $\ss < \ger b$ is consistent by the results in Section~\ref{sec:Laver}. We do not know whether the inequality $\rr < \ss$ is consistent. Thus we know that $\sr < \rr$ is consistent, but we do not know whether $\sr < \ss$ is consistent also.

\end{document}